%% file: main.tex
\pdfoutput=1
\documentclass[OJMO,manuscript,Unicode]{cedram}
\usepackage[all]{xypic}

\usepackage{booktabs}
\usepackage{bm}
\usepackage[nocomma]{optidef}
\usepackage{algpseudocode} 
\usepackage{algorithm}
\usepackage{etoolbox}
\usepackage{accents}
\usepackage{appendix}
\usepackage{multicol}
\usepackage{mdframed}
\usepackage{pifont}
\usepackage{tikz}
\usetikzlibrary{arrows.meta}
\newcommand{\cmark}{\ding{51}}%
\newcommand{\omark}{\ding{108}}%

\newtheorem{example}{Example}[section]
\newtheorem{theorem}{Theorem}
\newtheorem{lemma}{Lemma}
\newtheorem{corollary}{Corollary}
\newtheorem{definition}{Definition}
\newmdtheoremenv{mdlemma}{Lemma}

\title[Sensitivity analysis for changes of the constraint matrix in (MI)LP]{Sensitivity analysis for linear changes of the constraint matrix of a (mixed-integer) linear program}
\author{\firstname{Guillaume} \lastname{Derval}}
\address{Montefiore Institute, ULi\`ege, Belgium}
\email{gderval@uliege.be}
\author{\firstname{Damien} \lastname{Ernst}}
\address{Montefiore Institute, ULi\`ege, Belgium}
\email{dernst@uliege.be}
\author{\firstname{Quentin} \lastname{Louveaux}}
\address{Montefiore Institute, ULi\`ege, Belgium}
\email{q.louveaux@uliege.be}
\author{\firstname{Bardhyl} \lastname{Miftari}}
\address{Montefiore Institute, ULi\`ege, Belgium}
\email{bmiftari@uliege.be}


\begin{abstract} Understanding how the optimal value of an optimisation problem changes when its input data is modified is an old question in mathematical optimisation. This paper investigates the computation of the optimal values of a family of (possibly mixed-integer) linear optimisation problems in which the constraint matrix is subject to linear perturbations controlled by a scalar parameter that varies within a given interval. This is a largely unresolved question with the additional burden that the resulting value function may be largely irregular. 
We propose several bounding techniques that provide formal guarantees on the behaviour of the objective value across the entire parameter range. 
The proposed bounds rely on tools from robust optimisation, Lagrangian relaxation, and ad-hoc reformulations. Each method is assessed in terms of accuracy, precision, and computational performance. Experimental results on synthetic instances derived from Netlib and MIPLIB and on 126 mixed-integer instances from seven application domains show that the proposed bounding techniques effectively address this class of problems, delivering strong guarantees and good precision. In addition, we introduce an adaptive refinement algorithm that incorporates these bounds to compute an anytime approximation of the value function within a given error tolerance, and we analyse its computational performance.
\end{abstract}

\begin{document}
\maketitle

\input{TexFiles/introduction}

\input{TexFiles/problem_formalization}
\input{TexFiles/bounds}

\input{TexFiles/experiments}
\input{TexFiles/conclusion}

\section*{Disclosure statement}
The authors report there are no competing interests to declare.

\section*{Funding}
The authors gratefully acknowledge the support of the Public Service of Wallonia through the funding of the NKL project in the framework of the Recovery and Resilience Plan (PNRR), initiated and financed by the European Union, and of the SCK-CEN via the SCK-CEN SMR Chair at the University of Liège.

\section*{Availability of data and materials}
The dataset and codes to run the experiments are available on Zenodo \cite{zenodo_code}. The code archive contains a Snakemake \cite{10.12688/f1000research.29032.1} build file allowing to reproduce all experiments. The raw outputs are also available \cite{zenodo_output}.

\bibliographystyle{plain+eid}
\bibliography{ref}
\newpage
\appendix
\section{APPENDIX}
\input{TexFiles/applications}

\end{document}

%% file: TexFiles/introduction.tex
\section{Introduction}

One strength of linear programming models is the ability to analyze the variation of an optimal solution when we account for slight changes in the data. A popular example comes from modifying one cost coefficient in the objective. In this case, sensitivity analysis allows you to determine in which range the cost coefficient can vary while keeping the current optimal solution optimal. Similarly, if you assume the variation of one right-hand-side value, we can compute the interval in which the current optimal basis remains optimal. In this case, though, the optimal value varies linearly and the optimal dual variable provides the slope of such a linear variation. This can be very helpful for practitioners as it helps understand how robust their solution is and identify which parameters have the greatest impact on the outcome. As a result, linear programming not only provides an optimal solution but also valuable insights into how changes in the model's inputs affect that solution.

However the analysis becomes more complicated when we try to extend it to changing several coefficients simultaneously, or when we want to analyze the variation implied by the change of a matrix coefficient, especially when this matrix coefficient appears in the basis. 

Let us detail an application in which this may be natural to do though. Consider a linear optimization problem coming from modeling the investment of an electric infrastructure. To model such a problem, we typically want to simulate the production and consumption of electricity of a few devices over a long horizon. We therefore need to consider a time-indexed formulation that includes, over each period, the same constraints that repeat but include time-indexed variables depending on the time we consider them. For such a model, we typically need to make a few assumptions : the efficiency of the devices, the amount of production of the renewable energy, the amount of consumption, etc. If we look in particular at the assumed efficiency of a device, we can see that the parameter that we have assumed may appear several times in the problem, mostly with the same value, repeated for every time-period. So if this parameter is uncertain, changing it would involve changing by a same factor several coefficients of the matrix of the linear program. 

There are countless similar examples. We can consider the uncertain loss of a percentage of some product during transport, the uncertain ramp-up/down rates of some machine or electrical units, the unknown amount of food needed for livestock, the percentage of workers leaving a company each year and so on. 

From these examples, we want to propose a new model for studying the variation of the optimal value on a linear program. Building on the specific efficiency example, we assume that one parameter varies, but this single parameter may appear in several coefficients of the matrix $A$. More specifically, we consider a nominal linear (or mixed-integer) program with $n = n_c + n_i$ decision variables $\mathbf{x}$, of which $n_c$ are continuous and $n_i$ are integer, $m$ constraints with matrix $A \in \mathbb{R}^{m \times n}$ and right-hand side $\mathbf{b} \in \mathbb{R}^m$, and an objective vector $\mathbf{c} \in \mathbb{R}^n$:

\begin{align}
\begin{split}
    \min\quad &\mathbf{c}^t \mathbf{x}\\
    \text{ s.t } \quad&A\mathbf{x} \le \mathbf{b}\\
     &\mathbf{x} \in \mathbb{R}^{n_c} \times \mathbb{Z}^{n_i}.
\end{split}
\end{align} 
We aim at analyzing how the optimal value varies when one specific parameter $\lambda \in \mathbb{R} $ varies. We assume without loss of generality that the nominal value of $\lambda$ is 0. This parameter has a linear influence on a subset of the coefficients of the matrix $A$. More formally, we are interested in the following problem:
\begin{align}
\begin{split}
    \mathcal{P}(\lambda) \equiv \min\quad &\mathbf{c}^t \mathbf{x}\\
    \text{ s.t } \quad&A\mathbf{x} + \lambda D'\mathbf{x} \le \mathbf{b}\\
     &\mathbf{x} \in \mathbb{R}^{n_c} \times \mathbb{Z}^{n_i}.
\end{split}\label{eq:our_generic}
\end{align} 
where $D' \in \mathbb{R}^{m \times n}$ is a given matrix and $\lambda \in [\lambda_1, \lambda_2]$ is a scalar parameter, which represents the potential variation of the coefficients within some limits (it is not a random variable). We denote by $f(\lambda)$ the optimal value of $\mathcal{P}(\lambda)$, with the conventions $f(\lambda) = +\infty$ if $\mathcal{P}(\lambda)$ is infeasible and $f(\lambda) = -\infty$ if it is unbounded. The nominal optimal value is therefore given by $f(0)$ with $0\in[\lambda_1,\lambda_2]$. In this paper, we propose a novel approach consisting of bounding methods on $f(\lambda)$ that provide strong guarantees on the function’s behavior while remaining computationally efficient. First, we discuss the existing solutions in the literature.

\subsection{Related works}\label{sec:related}
Two closely related fields dealing with these types of problems are Sensitivity Analysis (SA) and Parametric Linear Programming (PLP). They differ based on the assumptions they make upon the modification of $\lambda$. Sensitivity analysis considers the effect of the small variation of a parameter around the optimum, whereas linear parametric programming assesses the effect of a parameter on the objective when it varies within a certain range. They both assess the impact of uncertainty on the objective. The methods discussed in this paper are agnostic of the amount of variation and can be applied in both fields.\\

A complementary field is Robust Optimisation (RO). In contrast to both SA and PLP, RO does not assess the impact of the modification on the objective. Indeed, while SA and PLP focus on the evolution of the optimal objective function given the changes in the parameter, RO techniques focus on finding  a solution that is robust to the modification, i.e. a sub-optimal solution that remains feasible for every change in parameter. The three methods, SA, PLP and RO, are complementary as they all try to deal with uncertainty. Several methods discussed in this paper use robust optimisation techniques to achieve the sensitivity analysis. \\

In its most generic form, the linear problem with a single scalar parameter $\lambda$, where the objective coefficients, the constraints and the right-hand side can be simultaneously modified, can be written as follows:
\begin{align}
\begin{split}
    \min\quad &(\mathbf{c} + \lambda \mathbf{c'})^t \mathbf{x}\\
    \text{ s.t } \quad&(A + \lambda D')\mathbf{x} \le \mathbf{b} + \lambda \mathbf{b'}\\
     &\mathbf{x} \in \mathbb{R}^{n_c} \times \mathbb{Z}^{n_i}
\end{split}\label{eq:generic1}
\end{align}
where $\lambda$ is a parameter that varies in a given range $[\lambda_1, \lambda_2]$; $\mathbf{c'}$, $D'$ and $\mathbf{b'}$ are the parameter's impact on the objective, constraint matrix and right-hand-side term, respectively. The end goal is to evaluate its optimal value $f(\lambda)$, over the range $[\lambda_1, \lambda_2]$.

Reoptimising the problem from scratch for a finite subset of values of $\lambda$ in $[\lambda_1, \lambda_2]$ can be computationally very costly, especially for big problems that take several minutes or hours to solve once (for a single $\lambda$) and does not provide any guarantees in between the computed points. Therefore, several approaches have been considered to mitigate this issue. 

Most of the literature focuses on one specific modification at a time; in general, either the objective (via $\mathbf{c'}$), or the right-hand side (RHS, via $\mathbf{b'}$), rarely the left-handside (LHS, via $\mathbf{D'}$) of the continuous version of the linear problem \eqref{eq:generic1}, i.e. where all the variables are continuous $\mathbf{x} \in \mathbb{R}^{n}$. Sensitivity analysis for mixed-integer problems is much less developed \cite{Geoffrion}, as most LP techniques rely on an optimal basis, which no longer characterizes the optimal solution of a MILP. Most works consider modifications of the right-hand side or of the objective \cite{Bowman01121972,Anderson2023}. Dawande and Hooker \cite{Dawande2000} also allow perturbations of the constraint matrix: from the branch-and-cut tree that proves the optimality of the nominal solution, they derive linear conditions under which a perturbation does not decrease the optimal value by more than a given amount. Multi-parametric programming computes the exact optimal value as a function of parameters that may appear in the constraint matrix of a MILP \cite{li2007mpmilp,charitopoulos2018mpmilp}, which is computationally demanding in that case, and Crema \cite{crema2002parametric} proposes a complete parametric analysis relative to the constraint matrix for 0--1 integer programs. 

\paragraph*{Modification of either the objective or the right-hand side.}

For the continuous linear case, let us consider $\mathbf{x^\star}$ and $\boldsymbol{\rho^\star}$ the primal and dual optimal solutions of the continuous unmodified problem, i.e. where $\lambda = 0$:

\noindent\begin{minipage}{.5\linewidth}
\begin{align*}
    \mathbf{x^\star} = \text{argmin}\quad &\mathbf{c}^t \mathbf{x}\\
    \text{ s.t } \quad&A\mathbf{x} \le \mathbf{b}\\
     &\mathbf{x}\geq \mathbf{0}
\end{align*}
\end{minipage}%
\begin{minipage}{.5\linewidth}
\begin{align*}
    \boldsymbol{\rho^\star} = \text{argmax}\quad &\mathbf{b^t} \boldsymbol{\rho}\\
    \text{ s.t } \quad&A^t\boldsymbol{\rho} \le \mathbf{c}\\
     &\boldsymbol{\rho}\leq \mathbf{0}.
\end{align*}
\end{minipage}

The modifications on the objective and the right-hand side are closely related as by dualising the problem, one can be transformed into the other. Note that when the modification is only performed on the objective, i.e. only $\mathbf{c'} \neq \mathbf{0}$, the previous primal optimal solution $\mathbf{x^\star}$ remains feasible and, similarly, when it is only performed on the right-hand side, i.e. only $\mathbf{b'} \neq \mathbf{0}$, the previous optimal dual solution $\boldsymbol{\rho^\star}$ remains dual feasible. This is what makes warm-starting the primal (resp. dual) simplex method from the previous optimal basis effective in these two cases. In contrast, a modification of the constraint matrix appears both in the primal constraints and in the dual constraints $A^t\boldsymbol{\rho} \le \mathbf{c}$: in general, neither $\mathbf{x^\star}$ (unless the perturbation only affects nonbasic columns) nor $\boldsymbol{\rho^\star}$ remains feasible. The methods used for these types of modifications can be categorized into three families \cite{jansen}: 
\begin{itemize}
    \item \textit{ones using optimal partitions}. The set of active (tight) constraints on the primal and on the dual form an optimal partition of the linear problem. From this partition, we can derive validity intervals upon $\lambda$. The methods derived from this family use an LP solver as a subroutine to solve the problem and iteratively find these partitions. 
    Adler and Monteiro \cite{Adler} introduce the first algorithm from this family. Berkelaar et al \cite{Berkelaar1997} introduce another algorithm based on the additional property that either the primal or the dual optimal set remains unchanged when $\mathbf{c'}$ or $\mathbf{b'}$ is non-zero.
    \item \textit{ones using optimal values}. The optimal values of the primal, dual and objective are used to build two auxiliary LP problems that give the range upon which solution stays optimal.
    \item \textit{ones using optimal basis}. The simplex algorithm returns the optimal basis related to a problem. These methods use warm starting and properties related to the basis to perform only a few iterations of the simplex to find another basis when the objective function changes. A two-part paper from Gass and Saaty \cite{obj_param1,obj_param2} uses a modified simplex method to retrieve the optimal solution for multiple changes in the coefficient of the objective function. Gal and Nedoma \cite{gal_multi} present an effective method for finding the regions that keep a basis optimal for multiple changes in the parameters for both types of modification using the simplex tableaux of multiple reoptimisation and graph theory to combine the tableaux.
\end{itemize}
Various works \cite{Geoffrion,bertsimas,jansen} make a summary of some techniques and conceptual foundation for post-optimality analysis for modifications on $\mathbf{c}$ and $\mathbf{b}$.

\paragraph*{Modifications on the constraint matrix.} 
It should be noted that when considering modifications $\mathbf{c'}$ and $\mathbf{b'}$, there always exists an equivalent problem (of the form \eqref{eq:generic1}) that encompasses these modifications inside its matrix $D'$, and such that its objective and right-hand side do not depend on $\lambda$. Indeed, introducing a variable $w$ fixed to $1$ and a free variable $t$, problem \eqref{eq:generic1} is equivalent to
\begin{align*}
    \min\quad &\mathbf{c}^t \mathbf{x} + t\\
    \text{ s.t } \quad&(A + \lambda D')\mathbf{x} - \lambda \mathbf{b'} w \le \mathbf{b}\\
    & \lambda \mathbf{c'}^t\mathbf{x} - t \le 0\\
    & w = 1,\ t \in \mathbb{R},\ \mathbf{x} \in \mathbb{R}^{n_c} \times \mathbb{Z}^{n_i},
\end{align*}
in which $\lambda$ only multiplies coefficients of the constraint matrix (at the optimum, $t = \lambda \mathbf{c'}^t\mathbf{x}$). The converse is not true: the modifications of the constraint matrix are more general. Classical local sensitivity analysis \cite{bertsimas} determines the ranges over which an optimal basis remains optimal when $\mathbf{b}$ or $\mathbf{c}$ is modified. The same reasoning applies to a modification of the constraint matrix only in a limited way. If $D'$ only modifies nonbasic columns, the basic solution is unchanged and the reduced costs of these columns vary linearly with $\lambda$, which gives a range of $\lambda$ over which the basis remains optimal. However, as soon as $D'$ modifies columns of the optimal basis $B$, both the basic solution and the reduced costs depend on $(A_B + \lambda D'_B)^{-1}$, where $A_B$ and $D'_B$ denote the columns of $A$ and $D'$ associated with the basic variables. This is a rational function of $\lambda$; the resulting conditions are only valid as long as the basis remains optimal, after which a new basis must be computed and inverted. This is the approach of the works discussed below \cite{Zuidwijk,lhs_param_algo,derval2025warmstart}. The robust envelope of Section~\ref{sec:robust_concav_env} uses the same reduced-cost conditions, but on a reformulation of the problem in which $\lambda$ only appears in the objective, so that these conditions become linear in $\lambda$. 
Gal \cite{Gal-1994} summarizes the existing techniques for modifications of the constraint matrix, mostly based on the Sherman--Morrison formulas, which update the inverse of a matrix when one of its elements \cite{Sherman}, or the elements of one of its columns or rows \cite{Sherman2}, are modified. Such a modification has rank one, i.e. it can be written $\mathbf{u}\mathbf{v}^t$ with $\mathbf{u}, \mathbf{v}$ two vectors, and the inverse of the modified matrix is
\begin{align*}
    (B + \lambda \mathbf{u}\mathbf{v}^t)^{-1} = B^{-1} - \frac{\lambda B^{-1}\mathbf{u}\mathbf{v}^t B^{-1}}{1 + \lambda \mathbf{v}^t B^{-1}\mathbf{u}}.
\end{align*}
Woodbury \cite{woodbury} generalizes this formula to a modification of rank $k$, written $UCV$ with $U \in \mathbb{R}^{m \times k}$, $C \in \mathbb{R}^{k \times k}$ and $V \in \mathbb{R}^{k \times m}$, at the price of inverting a $k \times k$ matrix that depends on $\lambda$. Applied to the matrix $A_B$ of an optimal basis, these formulas give the basic solution and the reduced costs as rational functions of $\lambda$, efficiently when the rank of $D'_B$ is small. They are however subject to the limitation discussed above: they only hold as long as the basis remains optimal. Moreover, the rank of $D'$ depends on the application, and it can be large. For instance, in time-indexed models such as the one described at the beginning of this section, modifying the efficiency of a machine typically affects a diagonal of the constraint matrix, with one coefficient per time period, each in a different row and column: even though the efficiency has the same value in all periods, the rank of $D'$ is the number of periods. Inverting the $\lambda$-dependent matrix $A_B + \lambda D'_B$ then becomes extremely complicated in general: its determinant and the entries of its inverse are polynomials in $\lambda$ whose degree can reach the rank of $D'_B$, and even determining the values of $\lambda$ at which this matrix becomes singular, or at which the basis stops being optimal, requires finding the roots of such polynomials. This $\lambda$-dependent matrix inversion is precisely what the bounds proposed in this paper avoid. In the application dataset of Section~\ref{sec:dataset}, the rank of $D'$ is never one: it ranges from 3 to 1430, with a median of 25 (see also Section~\ref{sec:rank}).  \\

More recently, Zuidwijk \cite{Zuidwijk} derives explicit local formulas to compute the evolution of the objective function and intervals on which these formulas are valid. They rewrite the basic inverse matrix as a function of $\lambda$ and recompute it for several values in the range. Their algorithm involves the finding of the optimal basis, computation of the range of $\lambda$ for which the basis stays optimal, the evaluation of the derived formulas using the said basis and reoptimisation when switching to another basis. Their formulas for computing the objective for a given basis require the computation of the eigenvalues of the matrices $A$ and $D'$ and leads to a polynomial problem with orders equal to the number of constraints. For big LPs, finding the solution of the polynomial problem can only be done using approximation methods. Khalipour et al. \cite{lhs_param_algo} also discuss an algorithm similar to Zuidwijk \cite{Zuidwijk}. However, they claim that their algorithm can be applied when working on larger problems. They use the Flavell and Salkin \cite{flavell1975approach} formula to compute an approximation of the basic inverse matrix.\\
All of the above-mentioned methods for the modification of the constraint matrix rely either on heavy computations, reoptimisations, approximations or are not sufficiently generic. As they rely on an optimal basis, none of them applies to mixed-integer problems. Additionally, we underline that even discretizing $[\lambda_1, \lambda_2]$, with a very fine granularity, at the price of heavy computations, may still lead to a false assessment of the behavior of $f(\lambda)$ over the interval. Indeed, as the matrix $D'$ can contain an arbitrary number of constraints, the resulting problem for a given $\lambda$ can behave in an unpredictable and sometimes erratic fashion in between two closely located points. 

\subsection{Our contributions}
In this paper, we propose methods to compute upper and lower bounds of problems given in \eqref{eq:our_generic}, for varying $\lambda \in [\lambda_1, \lambda_2]$. The bounding methods provide guarantees between the computed points. They can capture any erratic behaviour that might be missed by sampling approaches. They also give an indication of the evolution of the objective and can help practitioners focus on particular values of $\lambda$ within the interval that have an interesting or unconventional behaviour. The bounding methods proposed are based on three approaches: robust optimisation, Lagrangian relaxations and specific reformulations. Out of these approaches, we derive several bounding methods: ones that computes constant bounds, ones giving variables bounds depending on $\lambda$ and ones computing envelopes on the objective function. The envelopes are a combination of multiple variable bounds, proven optimal in the sense that no other variable bound of the same type is tighter,  that form an envelope around the objective function. Some methods add new degrees of freedom, and therefore, can lead to multiple variations. \\

For LPs, all the bounding methods provide both lower and upper bounds, some of them by being applied to the dual problem. Since sensitivity analysis of MILPs with respect to the constraint matrix is much less developed, the methods that apply to MILPs are of particular interest: three of them provide lower bounds (the coefficient-wise relaxation and the two Lagrangian bounds, using the dual solutions of the linear relaxation) and two provide upper bounds (the constant robust solution and the coefficient-wise restriction). Table~\ref{tab:methods} summarizes which bound applies to which class of problems. \\

To test our methods, we build two new datasets: a synthetic dataset of LPs and MILPs derived from the Netlib and MIPLIB libraries, in which either equality or inequality constraints are perturbed, and an application dataset of 126 MILP instances from seven application domains, in which the parameter has a physical or economic meaning. For each of these datasets, we study the bounding methods in terms of availability (whether a bound is available or not), error (with respect to the optimal value function) and timing. With the interval split into 10 subintervals, lower bounds are available almost everywhere (on more than $96\%$ of the points for the coefficient-wise relaxation, on every dataset), and so are upper bounds when inequality constraints are perturbed (more than $87\%$). Upper bounds are harder to obtain when equality constraints are perturbed (at most $64\%$ of the points for LPs and $40\%$ for MILPs), since a constant solution must then satisfy additional equality constraints, whatever the length of the interval. \\

In terms of error, measured as a fraction of the range of variation of the optimal value function, the most precise bounds on LPs have a median error of at most $0.06$, except for upper bounds when equality constraints are perturbed. On MILPs, the coefficient-wise bounds have a median error between $0.12$ and $0.28$, for a computation time of about one solve of the nominal problem per subinterval; on the application dataset, they are available on $99.8\%$ (lower bound) and $88.9\%$ (upper bound) of the points. We also show that the bounds do not depend on the rank of the perturbation. In addition to the bounding methods, we also introduce an adaptive refinement algorithm to compute these bounds with relatively small gaps and discuss its efficiency.

\paragraph*{Paper outline.}
In Section \ref{sec:problem}, we formalise and analyse the problem at hand. In Section \ref{sec:bounds}, we introduce all the bounding methods in the form of theorems and provide the associated proofs. A summary of the bounding methods is given in Table \ref{tab:methods}. Some methods work only on specific subsets of problems, like continuous problems or non-negative variables. In Section \ref{sec:expe}, we introduce two new datasets of problems to serve as a benchmark and perform three experiments. First, we benchmark the different bounds on the data problems in terms of availability, precision and timing for providing upper and lower bounds. Second, we study the influence of the rank of the perturbation. In the third experiment, we introduce an adaptive refinement algorithm that uses lower and upper bounds to compute the objective function and refines the parameter interval to minimize the gap between the bounds.

\begin{table}[H]
    \centering
    \begin{tabular}{llcccccc}
         & & & \multicolumn{2}{c}{LP} & \multicolumn{2}{c}{MILP} \\
         \cmidrule(r){4-5}\cmidrule(l){6-7}
         Approach & Type & Section & LB & UB & LB & UB & Cond. \\
         \toprule
         Robust & Constant & \ref{sec:constant} & \omark & \cmark &  & \cmark & \\
         & Linear in $\lambda$ &\ref{sec:variable} & \omark & \cmark &  &  & \\
         & Envelope & \ref{sec:robust_concav_env} & \omark & \cmark &  & \\
    Coeff. & Constant & \ref{sec:coeff_wise} & \cmark & \cmark & \cmark & \cmark & $\mathbf{x} \geq \mathbf{0}$\\
         Lagrangian 
                    & Bi-segment & \ref{sec:lag_flat} & \cmark & \omark & \cmark & & \\
                    & Bi-segment + coeff. & \ref{sec:lag_flat} & \cmark & \omark & \cmark & & $\mathbf{x} \geq \mathbf{0}$\\
         \bottomrule
    \end{tabular}
    \caption{Summary of all the bounding methods presented in this paper. \cmark: the bounding method is available for the category. \omark: available via dualisation.}
    \label{tab:methods}
\end{table}

\subsection*{Notations}

We denote a scalar with a standard-font symbol $a$, a vector with a bold lower case symbol $\mathbf{a}$, a matrix with an uppercase letter $A$, a vector space with a calligraphic uppercase letter $\mathcal{A}$. The $i^{th}$ element of a vector $\mathbf{a}$ is noted $a^{i}$. The $i^{th}$ row and $j^{th}$ column of a matrix $A$ are respectively given by $\mathbf{a}^{i, .}$ and $\mathbf{a}^{., j}$.  The element $i, j$ of a matrix $A$ is given by $a^{i,j}$.

%% file: TexFiles/problem_formalization.tex
\section{Problem formalisation}\label{sec:problem}
In this section, we provide a full formalisation of the considered problem.
Let us recall the problem $\mathcal{P}(\lambda)$ defined in \eqref{eq:our_generic}: 
\begin{align*}
    \mathcal{P}(\lambda)\equiv \min\quad &\mathbf{c}^t \mathbf{x}\\
    \text{ s.t } \quad&A\mathbf{x} + \lambda D'\mathbf{x} \le \mathbf{b}\\
     &\mathbf{x}\in \mathbb{R}^{n_c}\times \mathbb{Z}^{n_i}.
\end{align*} 
Its optimal value, as a function of $\lambda$, is $f(\lambda)$.

We can modify this problem without losing any generality: we decompose the constraint matrix $A\in \mathbb{R}^{m\times n}$ in two matrices $A_1$ and $A_2$ of respective size $m_1\times n$ and $m_2\times n$, with $m_1+m_2=m$. The matrix $A_1$ encompasses all the constraints upon which there are no modifications depending on the external parameter $\lambda$, while $A_2$ contains the constraints affected by the modification $\lambda D'$, and $D \in \mathbb{R}^{m_2 \times n}$ denotes the corresponding rows of $D'$ (the other rows of $D'$ are zero). Similarly, we divide $\mathbf{b}$ in $\mathbf{b_1}$ and $\mathbf{b_2}$ of respective size $m_1\times 1$ and $m_2\times 1$. As we can have $m_1=0$, we indeed do not lose any generality.
This results in the following equivalent problem: 
\begin{align}
\begin{split}
    \mathcal{P}(\lambda)\equiv\min\quad &\mathbf{c}^t \mathbf{x}\\
    \text{ s.t } \quad&A_1\mathbf{x} \le \mathbf{b_1}\\
    & A_2\mathbf{x}+\lambda D\mathbf{x}\le \mathbf{b_2}\\
    & \mathbf{x} \in \mathbb{R}^{n_c}\times \mathbb{Z}^{n_i}.
\end{split}\label{eq:separated}
\end{align}

For a fixed value $\bar \lambda$, an optimal solution of $\mathcal{P}(\bar \lambda)$ is denoted $\mathbf{x}_{\mathcal{P}(\bar \lambda)}^*$ and in the case of a purely continuous problem (i.e. $n_i=0$) an optimal dual solution $\boldsymbol{\rho}_{\mathcal{P}(\bar \lambda)}^*$. For the sake of readability, when explicit, we may omit the subscript of $\mathbf{x}_{\mathcal{P}(\bar \lambda)}^*$ and $\boldsymbol{\rho}_{\mathcal{P}(\bar \lambda)}^*$ and write $\mathbf{x}^*$ and $\boldsymbol{\rho}^*$.

The goal of this paper is to find $ub_{\lambda_1, \lambda_2}(\lambda)$ and $lb_{\lambda_1, \lambda_2}(\lambda)$, respectively an upper bound and lower bound of the optimal value on the problem written in \eqref{eq:separated}, i.e. $lb_{\lambda_1, \lambda_2}(\lambda) \le
f(\lambda)\le ub_{\lambda_1, \lambda_2}(\lambda), \forall \lambda \in [\lambda_1, \lambda_2].
\label{eq:lbub}$
In general, $f(\lambda)$ has no desirable properties. Depending on $A_1, A_2, D, \mathbf{b_1}$ and $\mathbf{b_2}$, it is generally non-convex, non-concave and non-differentiable. As the problem may be unbounded or infeasible for some $\lambda$, it can even be non-continuous. In the following, we illustrate these behaviours with two examples.
\begin{example}
    Let us consider an LP with free variables $\mathbf{x} \in \mathbb{R}^n$, no unperturbed constraints ($m_1 = 0$), a nonzero objective $\mathbf{c} \neq \mathbf{0}$, $D = -A_2$ and $\lambda \in [0, 2]$. The problem $\mathcal{P}(\lambda)$ becomes
\begin{align*}
    \min\quad &\mathbf{c}^t \mathbf{x}\\
    \text{ s.t }& (1-\lambda) A_2\mathbf{x}\le \mathbf{b_2},\\
    & \mathbf{x} \in \mathbb{R}^n.
\end{align*}
For $\lambda = 1$, the constraints reduce to $\mathbf{0} \le \mathbf{b_2}$, which does not depend on $\mathbf{x}$. Hence, if $b_2^i < 0$ for at least one $i$, $\mathcal{P}(1)$ is infeasible and $f(1) = +\infty$. Otherwise, every $\mathbf{x} \in \mathbb{R}^n$ is feasible and, since $\mathbf{c} \neq \mathbf{0}$, $\mathcal{P}(1)$ is unbounded and $f(1) = -\infty$. In both cases, $f$ is not finite at $\lambda = 1$, whatever its values for $\lambda \neq 1$.
\end{example}

\begin{example}
In this example, we showcase the erratic behaviour of $f(\lambda)$. Let us consider the following problem, 

\begin{align}
\begin{split}
\mathcal{P}_{toy}(\lambda) \equiv 
    \min\quad &\begin{pmatrix}2  & -2\end{pmatrix} \begin{pmatrix}
        x \\ y
    \end{pmatrix}\\
    \text{ s.t } \quad& \begin{pmatrix}
        -2 & 2 \\
        -1 & 0
    \end{pmatrix}
    \begin{pmatrix}
        x \\ y
    \end{pmatrix}
    \leq \begin{pmatrix}
        4 \\ 1
    \end{pmatrix}\\
    & \begin{pmatrix}
        2 & 1 \\ -2 & -3 \\ 2 & 2 \\ -1 & -4
    \end{pmatrix}\begin{pmatrix}
        x \\ y
    \end{pmatrix} + \lambda \begin{pmatrix}
        -1 & -4 \\ 0 & 4 \\ -4 & -3 \\ 2 & 4
    \end{pmatrix}\begin{pmatrix}
        x \\ y
    \end{pmatrix}
    \leq 
    \begin{pmatrix}
        4 \\ 2 \\ 0 \\ 2
    \end{pmatrix}\\
    \forall \lambda \in [-10, 9].
\end{split}\label{example:toy}
\end{align}

We denote the associated optimal objective function by $f_\text{toy}(\lambda)$. Note that for this particular example we do not impose $x, y \ge 0$.

To highlight the difficulty of predicting the behaviour of the objective function of $f_\text{toy}(\lambda)$ through sampling and choosing the set size for sampling, we compute $f_\text{toy}(\lambda)$ for the same range of $\lambda$ values whilst using two different levels of granularity. On the one hand, we use a fine discretisation step of $0.01$ over $\lambda \in [-10, 9]$, shown by the blue curve in Figure \ref{fig:toy_problem}. As we can see, the value of $f(\lambda)$ varies strongly in a neighbourhood located after $\lambda = 0.5$ and is neither concave nor convex. On the other hand, we use a coarse discretisation step of $1$, shown by the orange curve on the same figure. As we can see, with this discretisation step, we would have missed the erratic behaviour of the function after 0, illustrated by the blue peak. While the difference between the two steps is rather large in the context of this toy example, this problem can also occur in much larger problems at much finer scales. 
\begin{figure}[H]
    \centering
    \includegraphics[width=0.65\textwidth]{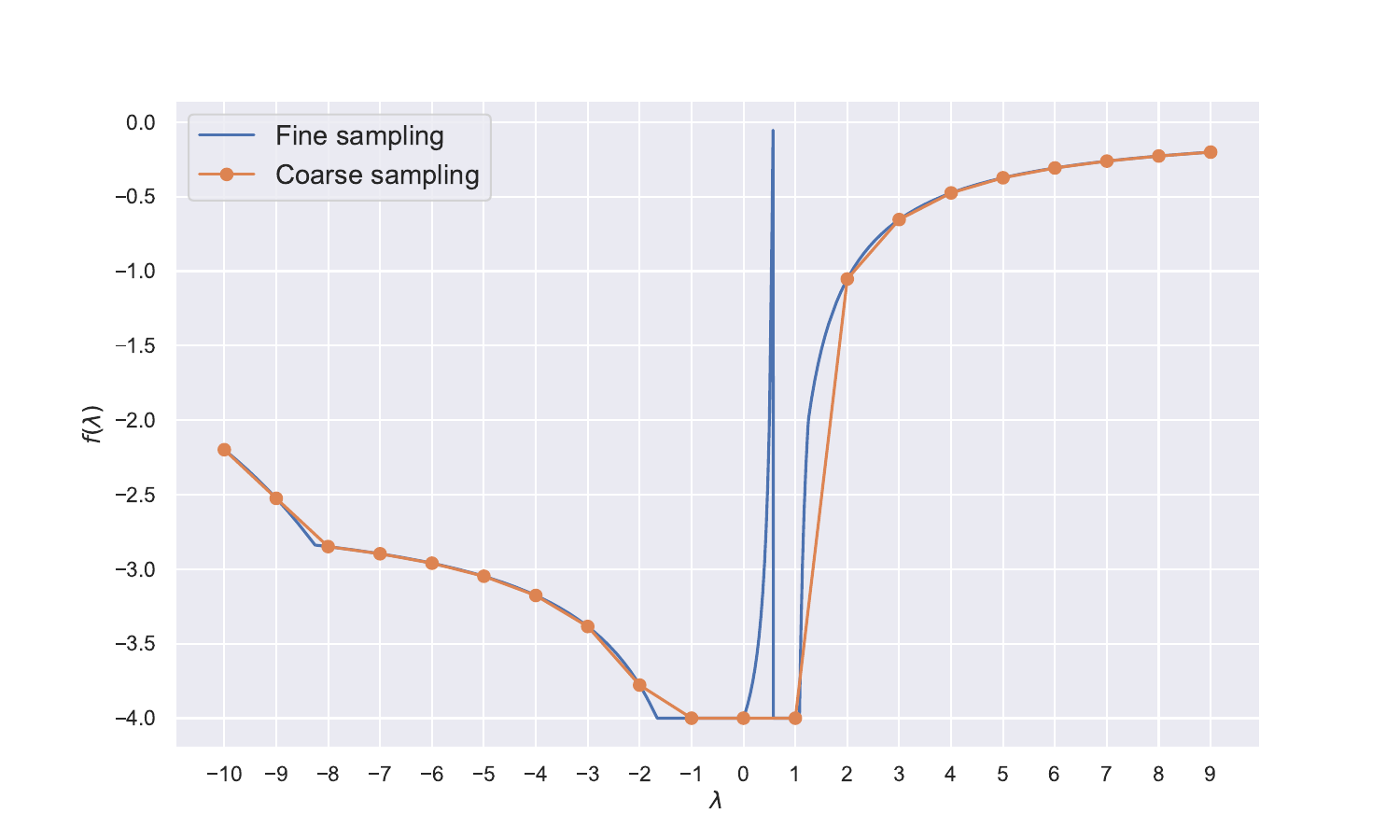}
    \caption{Plot of $f_{\text{toy}}(\lambda)$ between $[-10, 9]$. The orange line with dots is a coarse sampling of the function, made for each integer $\lambda \in \{-10, -9, \dots, 9\}$, and linearly interpolated between these points.}
    \label{fig:toy_problem}
\end{figure}

In Figure \ref{fig:feasible_space}, we show the evolution of the feasible space for three values of $\lambda$ in the interval $[0,1]$ where $f_\text{toy}(\lambda)$ shows an erratic behaviour. We see that the feasible space changes constantly due to some constraints switching from tight to non-tight and vice-versa, with small changes in $\lambda$.

\begin{figure}[h!]
\centering
\includegraphics[width=\textwidth]{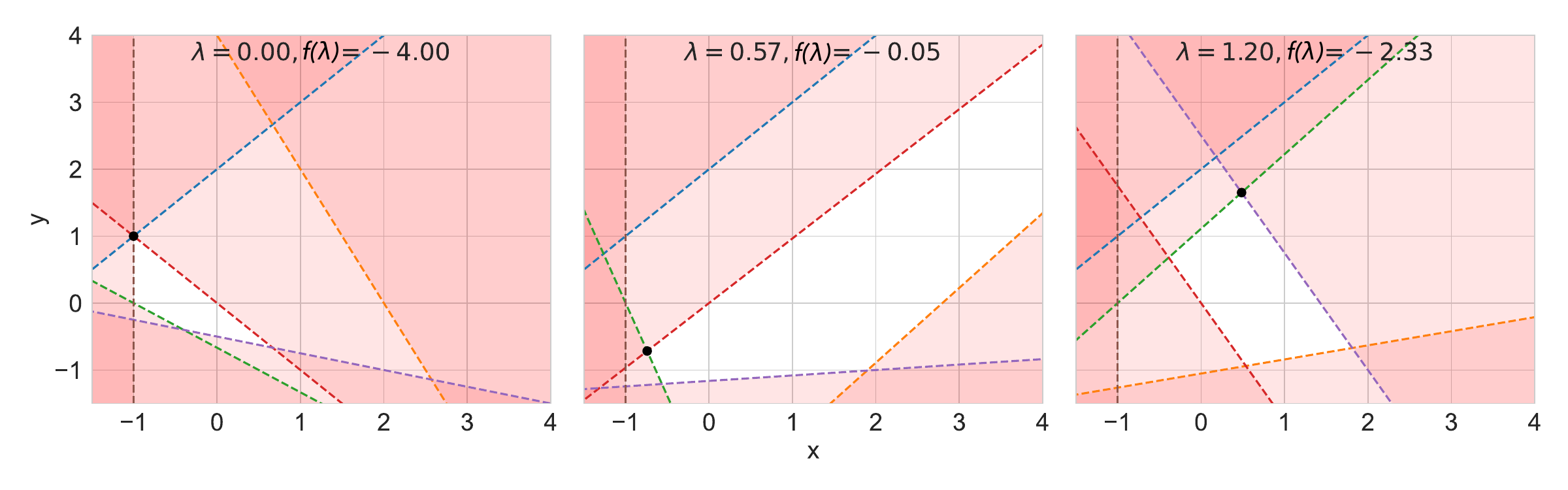}
\caption{The feasible space of the problem given in Equation \eqref{example:toy} for different values of $\lambda$. Each inequality partitions the space into two parts: a feasible space, coloured in white, and a non-feasible space, coloured in red. The optimum is shown as a black dot.}
\label{fig:feasible_space}
\end{figure}

\end{example}

%% file: TexFiles/bounds.tex
\section{Bounding methods formulations}\label{sec:bounds}
    In this section, we show how to compute bounds on the optimal value function $f(\lambda)$. We present our three approaches:
\begin{itemize}
    \item Robust optimisation: we reformulate the inner problem using several robust optimisation bounding methods separated in three types: 
    \begin{itemize}
        \item one consisting of a constant solution for a given range,
        \item a second solution linearly dependent on $\lambda$,
        \item a third one takes a combination of affinely dependent solution, to form an envelope around the objective function.
    \end{itemize}
    These techniques provide upper bounds. On LP problems, they can also be applied on the dual to provide lower bounds. The constant bound can be applied to MILP to provide upper bounds.
    \item Coefficient-wise relaxation and tightening: we reformulate the $\lambda D$ part of the modification constraints into a single matrix that do not depend on $\lambda$. This is done by considering each coefficient of the $D$ matrix independently, hence the \textit{coefficient-wise} name. These bounding methods provide upper and lower bounds, even if used only on the primal. They however require the variables to be nonnegative. These methods can be applied to MILP.
    \item Lagrangian relaxation: we dualise the constraints linked to the modification, $A_2\mathbf{x}+\lambda D\mathbf{x}\le \mathbf{b_2}$. Choosing the Lagrange multipliers as optimal dual solutions at the ends of the interval, we derive a bisegment bound, possibly combined with the coefficient-wise relaxation. These techniques provide lower bounds. On LP problems, they can also be applied on the dual problem to provide upper bounds. They also provide lower bounds for MILPs, using the optimal dual solutions of the linear relaxation as Lagrange multipliers.
\end{itemize}
These techniques are discussed in the next section on the primal. Their application to the dual problem leads to the opposite type of bound. In other words, if the technique applied on the primal gives a lower bound $lb(\lambda)$, resp. upper bound $ub(\lambda)$, its application on the dual yields an upper bound $ub(\lambda)$, resp. lower bound $lb(\lambda)$ on $f(\lambda)$. 

\subsection{Constant robust solution}\label{sec:constant}
The first approach taken consists of sacrificing the guaranteed optimality of the solution for a feasible solution on the whole interval $[\lambda_1, \lambda_2]$. Therefore, we turn to robust optimisation in order to search for a feasible solution that provides the tightest bound on this interval.

An upper bound can be obtained by finding a solution that satisfies $(A_2 + \lambda D) \mathbf{x} \leq \mathbf{b}_2\ \forall \lambda \in [\lambda_1, \lambda_2]$. The corresponding robust optimisation problem can be written as follows and gives an upper bound on $f(\lambda)$ in the range  $[\lambda_1, \lambda_2]$:

\begin{mini}
{x}
{\mathbf{c}^t\mathbf{x}}
{\label{constant_robust_1}}
{\mathcal{P}^{\text{CR}}_{\lambda_1, \lambda_2} \equiv }
\addConstraint{A_1 \mathbf{x}}{\leq \mathbf{b_1}}{}
\addConstraint{(A_2 + \lambda D) \mathbf{x}}{\leq \mathbf{b_2}}{\ \forall \lambda \in [\lambda_1, \lambda_2]}
\addConstraint{\mathbf{x}}{\in \mathbb{R}^{n_c} \times \mathbb{Z}^{n_i}.}
\end{mini}

This new linear reformulation of the initial problem is unusual in the sense that it possesses an infinite number of constraints, one for each $\lambda \in [\lambda_1, \lambda_2]$. It is possible to reformulate the problem in a short finite optimisation problem.

\begin{theorem}
    The following linear problem is equivalent to problem \eqref{constant_robust_1} and provides an upper bound for $f(\lambda)$ $\forall \lambda \in [\lambda_1, \lambda_2]$:
    \begin{mini}
{\mathbf{x}}
{\mathbf{c}^t\mathbf{x}}
{\label{eq:theorem2}}
{\mathcal{P}^{\text{CR}}_{\lambda_1, \lambda_2} \equiv }
\addConstraint{A_1 \mathbf{x}}{\leq \mathbf{b_1}}{}
\addConstraint{A_2 \mathbf{x} + \lambda_1D\mathbf{x}}{\leq \mathbf{b_2}}{}
\addConstraint{A_2 \mathbf{x} + \lambda_2D\mathbf{x}}{\leq \mathbf{b_2}}{}
\addConstraint{\mathbf{x}}{\in \mathbb{R}^{n_c} \times \mathbb{Z}^{n_i}.}{}
\end{mini}
\label{th:constant}
\end{theorem}
\begin{proof}
    Any feasible solution $\mathbf{x}$ to $\mathcal{P}^{\text{CR}}_{\lambda_1, \lambda_2}$ must respect the following condition:
    \begin{align}
        (A_2 + \lambda D) \mathbf{x} &\leq \mathbf{b_2}\ \forall \lambda \in [\lambda_1, \lambda_2]\ .
    \end{align}
    This can be reformulated using as its robust counterpart
    \begin{align}
        \max_{\lambda \in [\lambda_1, \lambda_2]} (A_2 + \lambda D) \mathbf{x} \leq \mathbf{b_2}
    \end{align}
    which is equivalent to 
    \begin{align}
        A_2 \mathbf{x} + \max_{\lambda \in [\lambda_1, \lambda_2]} \lambda D\mathbf{x} \leq \mathbf{b_2}, 
    \end{align}
    where the $\max$ operator is here applied component-wise.
    
    If we have $\mathbf{d}^{i, .}\mathbf{x} \geq 0$ (resp. $\leq 0$), then $\lambda = \lambda_2$ (resp. $\lambda = \lambda_1$) is an optimal solution to $\max_{\lambda \in [\lambda_1, \lambda_2]} \lambda \mathbf{d}^{i, .}\mathbf{x}$.  Satisfying these two cases is thus equivalent to the following set of constraints:
    \begin{align}\label{eq:robust_const_proof}
        \begin{cases} 
            A_2 \mathbf{x} + \lambda_1D\mathbf{x} \leq \mathbf{b_2}\\
            A_2 \mathbf{x} + \lambda_2D\mathbf{x} \leq \mathbf{b_2}
        \end{cases}
    \end{align}

    For each component $i$ of the vector $\mathbf{b_2}$, one of these constraints is redundant. 
    Hence, Problem \eqref{constant_robust_1} is equivalent to \eqref{eq:theorem2}.
\end{proof}

This reformulation has $n$ variables and $m_1+2m_2$ constraints and works for MILPs. It should be noted that for two given values of $\lambda_1$ and $\lambda_2$, the feasible space respecting both $A_2\mathbf{x} + \lambda_1 D\mathbf{x} \leq \mathbf{b_2}$ and $A_2\mathbf{x} + \lambda_2 D\mathbf{x} \leq \mathbf{b_2}$ can be empty if at least two constraints are conflicting.

\begin{example}
\label{example:robust-1}
The following problem is a very simple case where constraints can be conflicting:
\begin{align}
\begin{split}
    \min_{x, y} &\begin{pmatrix}
        -2 & -2
    \end{pmatrix}
    \begin{pmatrix}
        x \\ y
    \end{pmatrix}\\
    \text{s.t. } & \begin{pmatrix}
        3 & 1\\ 0 & -1
    \end{pmatrix}\begin{pmatrix}
        x \\ y
    \end{pmatrix} \leq \begin{pmatrix}
        3 \\ 3
    \end{pmatrix}\\
    &\begin{pmatrix}
        -5 & -2 \\
        1 & 4
    \end{pmatrix}
    \begin{pmatrix}
        x \\ y
    \end{pmatrix}+\lambda \begin{pmatrix}
        -3 & -2 \\
        -3 & 0
    \end{pmatrix} \begin{pmatrix}
        x \\ y
    \end{pmatrix} \leq
    \begin{pmatrix}
        0 \\ -3
    \end{pmatrix}\\
    & \forall \lambda \in [-2, 2]\ .
\end{split}\label{eq:example-robust-1}
\end{align}

Its objective depending on $\lambda$ is shown in Figure \ref{fig:example-robust-2}.

By applying the Theorem \ref{th:constant}, an upper bound is given by, 

\begin{align}
    \begin{split}
        \min_{x, y} & \begin{pmatrix}
            -2 & -2
        \end{pmatrix}\begin{pmatrix}
            x \\ y
        \end{pmatrix}\\
        \text{s.t. } & \begin{pmatrix}
            3 & 1\\0 & -1\\ 1 & 2 \\ 7 & 4 \\ -11 & -6 \\ -5 & 4
        \end{pmatrix}
        \begin{pmatrix}
            x \\ y
        \end{pmatrix}\leq
        \begin{pmatrix}
            3\\3 \\0 \\0 \\ -3 \\ -3
        \end{pmatrix}
    \end{split}\label{ex:extended_constant}
\end{align}
We show by applying Farkas' lemma that the problem \eqref{ex:extended_constant} is infeasible. Farkas' lemma states that for a given problem $\min c^t\mathbf{x} \text{ s.t. } A\mathbf{x} \le \mathbf{b}$ if we find a vector $\mathbf{u} \geq \mathbf{0}$ such that $\mathbf{u}^t A = \mathbf{0}$ and $\mathbf{u}^t \mathbf{b} < \mathbf{0}$ then this problem is infeasible. If we consider $\mathbf{u} = [2, 0, 0, 7, 5, 0]$, we indeed have $\mathbf{u}^t A = \mathbf{0}$ and $\mathbf{u}^t \mathbf{b} = -9$.
\end{example}

\subsection{Variable robust solution, affine on $\lambda$}\label{sec:variable}
A second possibility is to find solutions that vary linearly in $\lambda$. We perform a reformulation of all the variables $\mathbf{x}$ by replacing them with $\mathbf{y}+\lambda \mathbf{z}$, a linear function of $\lambda$. This transformation allows us to find a solution that adapts to a change in the value of $\lambda$.

We aim to find an upper bound $h^{AR}(\lambda)$ on the function $f(\lambda)$ of the form $\mathbf{c}^t(\mathbf{y}+\lambda \mathbf{z})$. The $AR$ superscript stands for \textbf{a}ffine \textbf{r}obust.
To ensure that $h^{AR}(\lambda)$ is indeed an upper bound, the newly introduced variables $\mathbf{y}$ and $\mathbf{z}$ must satisfy the following conditions: 
\begin{align}
    A_1 (\mathbf{y} + \lambda \mathbf{z}) & \leq \mathbf{b_1} & \forall \lambda \in [\lambda_1, \lambda_2]\label{ctr_var_1}\\
    (A_2 + \lambda D) (\mathbf{y} + \lambda \mathbf{z}) & \leq \mathbf{b_2} & \forall \lambda \in [\lambda_1, \lambda_2]\label{ctr_var_2}\\
    \mathbf{y} + \lambda \mathbf{z} &\in \mathbb{R}^{n_c} \times \mathbb{Z}^{n_i} & \forall \lambda \in [\lambda_1, \lambda_2]\label{integr}
\end{align}

The conditions \eqref{integr} are complex to satisfy if the problem is mixed-integer. In the remaining of this  section and the following one, we thus focus on linear problems, i.e. with $n_i = 0$, and the last set of constraints simplifies to $\mathbf{y}, \mathbf{z} \in \mathbb{R}^{n_c}$. 

As in the previous case, we can rewrite the remaining infinite number of constraints to a finite albeit more constrained form and obtain the following Theorem.

\begin{theorem}
    Any $\mathbf{y}$ and $\mathbf{z}$ satisfying the constraints
\begin{align}
    \begin{cases}
        A_1\mathbf{y} + \lambda_1 A_1\mathbf{z} \leq \mathbf{b_1}\\
        A_1\mathbf{y} + \lambda_2 A_1\mathbf{z} \leq \mathbf{b_1}\\
        (A_2 + \lambda_1 D) (\mathbf{y} + \lambda_1 \mathbf{z}) \leq \mathbf{b_2}\\
        (A_2 + \lambda_2 D) (\mathbf{y} + \lambda_2 \mathbf{z}) \leq \mathbf{b_2}\\
        A_2\mathbf{y} + D\mathbf{z}\lambda_1\lambda_2+(D\mathbf{y}+A_2\mathbf{z})\frac{\lambda_1+\lambda_2}{2} \leq \mathbf{b_2}
    \end{cases}
    \label{eq:affinespace}
\end{align}
also satisfy both equations \eqref{ctr_var_1} and \eqref{ctr_var_2}.

If $\mathcal{P}(\lambda)$ is a linear problem ($n_i=0$), any $\mathbf{y}$ and $\mathbf{z}$ satisfying these conditions provide an upper bound in the form $h^{AR}(\lambda)=\mathbf{c}^t(\mathbf{y}+\lambda \mathbf{z}) \geq f(\lambda)$ $\forall \lambda \in [\lambda_1, \lambda_2]$.

    \label{theorem:robust-line}
\end{theorem}

\begin{proof}
    We write the robust counterpart in order to obtain a finite number of constraints. Let us first focus on constraints \eqref{ctr_var_1}:
    \begin{align}
        A_1 (\mathbf{y} + \lambda \mathbf{z}) \leq \mathbf{b_1}\ \forall \lambda \in [\lambda_1, \lambda_2]
    \end{align}
    which can be rewritten as
    \begin{align}
        \max_{\lambda \in [\lambda_1, \lambda_2]}A_1 (\mathbf{y} + \lambda \mathbf{z}) \leq \mathbf{b_1}
    \end{align}
    and
    \begin{align}
        A_1\mathbf{y} + \max_{\lambda \in [\lambda_1, \lambda_2]}\lambda A_1\mathbf{z} \leq \mathbf{b_1}
    \end{align}
    As the max operator is applied component-wise, for every constraint in $\max_{\lambda \in [\lambda_1, \lambda_2]}\lambda A_1\mathbf{z}$, the maximum is either achieved for $\lambda=\lambda_1$ or $\lambda=\lambda_2$. 
    \begin{align}
        A_1\mathbf{y} + \max_{\lambda \in [\lambda_1, \lambda_2]}\lambda A_1\mathbf{z} \leq \mathbf{b_1}
         \equiv \begin{cases}
            A_1\mathbf{y} + \lambda_1 A_1\mathbf{z} \leq \mathbf{b_1}\\
            A_1\mathbf{y} + \lambda_2 A_1\mathbf{z} \leq \mathbf{b_1}
        \end{cases}.\label{rl_p1}
    \end{align}
    Hence, any $\mathbf{x}$ satisfying \eqref{rl_p1} also satisfies \eqref{ctr_var_1} and conversely.
    The case of \eqref{ctr_var_2} is more complex as its left-hand side is a quadratic function of $\lambda$. 
    \begin{align}
        & (A_2 + \lambda D) (\mathbf{y} + \lambda \mathbf{z}) = A_2\mathbf{y} + \lambda (D\mathbf{y} + A_2\mathbf{z}) + \lambda^2 D\mathbf{z} \leq \mathbf{b_2} \ \forall \lambda \in [\lambda_1, \lambda_2]
    \end{align}
    Again, we rewrite it using the robust counterpart as a maximization
    \begin{align}
         & \max_{\lambda\in [\lambda_1, \lambda_2]} A_2\mathbf{y} + \lambda (D\mathbf{y} + A_2\mathbf{z}) + \lambda^2 D\mathbf{z} \leq \mathbf{b_2}\\
        := & \max_{\lambda\in [\lambda_1, \lambda_2]} g(\lambda) \leq \mathbf{b_2}.
    \end{align}
    The function $g(\lambda)$ is a quadratic function of $\lambda$. The maximum of the quadratic equation can be found analytically but is itself non-linear in $z$. Instead of directly using $g(\lambda)$, we propose using a piecewise-linear upper bound of this function in the above condition; this allows us to relax \eqref{ctr_var_2} into a linear program in $\mathbf{y}$ and $\mathbf{z}$. 
    
    For this, we can use the following lemma:
    \begin{mdframed}[%
        topline=false,%
        rightline=false,%
        bottomline=false,
        leftmargin=2em
    ]
    \begin{lemma}
        Given a quadratic function $q(x)=ax^2+bx+c$. The maximum of $q(x)$ over $x_1 \leq x \leq x_2$ is upper bounded by 
        \begin{align*}
        \max \begin{cases}
            ax_1^2+bx_1+c\\
            ax_2^2+bx_2+c\\
            ax_1x_2 + b\frac{x_1+x_2}{2} + c.
        \end{cases}
        \end{align*}
    \end{lemma}
    \end{mdframed}
     This lemma is proved in Annex \ref{annex:lemma}. We apply the lemma on $g(\lambda)$ and we obtain that 
    \begin{align}
        \max_{\lambda\in [\lambda_1, \lambda_2]} g(\lambda) \leq \max \begin{cases}
        (A_2 + \lambda_1 D) (\mathbf{y} + \lambda_1 \mathbf{z})\\
        (A_2 + \lambda_2 D) (\mathbf{y} + \lambda_2 \mathbf{z})\\
        A_2\mathbf{y} + D\mathbf{z}\lambda_1\lambda_2+(D\mathbf{y}+A_2\mathbf{z})\frac{\lambda_1+\lambda_2}{2}
    \end{cases}
    \end{align}
    Imposing the following is thus sufficient to obtain $\max_{\lambda\in [\lambda_1, \lambda_2]} g(\lambda) \leq \mathbf{b_2}$:
    \begin{align}
    \begin{cases}
        (A_2 + \lambda_1 D) (\mathbf{y} + \lambda_1 \mathbf{z}) \leq \mathbf{b_2}\\
        (A_2 + \lambda_2 D) (\mathbf{y} + \lambda_2 \mathbf{z}) \leq \mathbf{b_2}\\
        A_2\mathbf{y} + D\mathbf{z}\lambda_1\lambda_2+(D\mathbf{y}+A_2\mathbf{z})\frac{\lambda_1+\lambda_2}{2} \leq \mathbf{b_2}\\
    \end{cases}.\label{rl_p2}
    \end{align}
    The conditions \eqref{rl_p1} and \eqref{rl_p2} are thus sufficient for $\mathbf{y}, \mathbf{z}$ to respect constraints \eqref{ctr_var_1} and \eqref{ctr_var_2}.
\end{proof}
Note that the converse is not true: the set of constraints \eqref{eq:affinespace} is only a sufficient condition for $\mathbf{y}$ and $\mathbf{z}$ to define an affine robust solution.
We can derive an upper bound for $f(\lambda)$ by using Theorem \ref{theorem:robust-line}.
\begin{corollary}
    Consider  the set
\begin{alignat}{2}
    S^{AR}_{\lambda_1,\lambda_2} =\{ \quad && \mathbf y, \mathbf z\in \mathbb R^n\mid A_1\mathbf{y} + \lambda_1 A_1\mathbf{z}&\leq \mathbf{b_1}\notag\\
&&A_1\mathbf{y} + \lambda_2 A_1\mathbf{z}&\leq \mathbf{b_1}\notag\\
&&(A_2 + \lambda_1 D) (\mathbf{y} + \lambda_1 \mathbf{z})&\leq \mathbf{b_2} \label{SAR}\\
&&(A_2 + \lambda_2 D) (\mathbf{y} + \lambda_2 \mathbf{z})&\leq \mathbf{b_2}\notag\\
&& A_2\mathbf{y} + D\mathbf{z}\lambda_1\lambda_2+(D\mathbf{y}+A_2\mathbf{z})\frac{\lambda_1+\lambda_2}{2} & \leq \mathbf{b_2}\quad \}.\notag
\end{alignat}
For any $(\mathbf y,\mathbf z)\in S^{AR}$, the function
$h^{AR}(\lambda) = \mathbf{c}^t(\mathbf y + \lambda \mathbf z) \geq f(\lambda)$ for all $\lambda \in [\lambda_1, \lambda_2]$.
\label{cor:AR}
\end{corollary}
When it is clear from the context, we may refer to $S^{AR}_{\lambda_1,\lambda_2}$ as simply $S^{AR}.$    
There can be a large number of $\mathbf{y}$ and $\mathbf{z}$ in the set $S^{AR}$ from Corollary \ref{cor:AR}. 
The question is to select them in order to have the tightest possible bound. We may fix  $\lambda$ to a specific value $\mu\in[\lambda_1,\lambda_2]$ and obtain the tightest possible bound by optimizing the problem 
\begin{align}
    \mathcal P^{AR}_{\lambda_1,\lambda_2}(\mu) \equiv \text{min } \mathbf{c}^t(\mathbf y+\mu \mathbf z) \text{ subject to }(\mathbf y,\mathbf z)\in S^{AR}_{\lambda_1,\lambda_2}.
    \label{eq:h}
\end{align}
An optimal solution $(\mathbf y,\mathbf z)$ to $\mathcal P^{AR}_{\lambda_1,\lambda_2}(\mu)$ provides the upper bound $h^{AR}_{\mathbf y,\mathbf z}(\lambda)$ over the entire interval $[\lambda_1,\lambda_2]$ but there is no guarantee that, by doing so, we obtain a tight bound for the rest of the interval outside of the specific value $\mu$. In the next paragraph, we describe strategies to derive a tight bound.

\subsubsection*{Ways to select $\mathbf{y}$ and $\mathbf{z}$}

There are several degrees of freedom in selecting $\mathbf y$ and $\mathbf z$ from $S^{AR}$ because it doubles the size of the space compared to the initial variables $\mathbf x$. For instance, any constant robust solution $\mathbf{x}$ can be mapped to a solution in $S^{AR}$.

\begin{theorem}
    Any constant robust solution $\mathbf{x}$ provides a solution $\mathbf{y}=\mathbf{x}, \mathbf{z}=\mathbf{0}$ for eq. \eqref{eq:affinespace}.

    \label{th:maprobustaffine}
\end{theorem}
\begin{proof}
    Enforcing $\mathbf{z}=\mathbf{0}$ in eq. \eqref{eq:affinespace} provides the same equations as \eqref{eq:theorem2} defining the space of constant robust solution (after renaming $\mathbf{x}$ to $\mathbf{y})$, with the following additional constraint:
    \begin{align}
        A_2 \mathbf{y} + \frac{\lambda_1+\lambda_2}{2} D \mathbf{y} \leq \mathbf{b}_2.
    \end{align}

    This constraint is redundant with the other constraints (it is a linear combination of the two constraints related to $A_2$), hence the two solution spaces are the same.
\end{proof}

Of course, many more solutions may exist with $\mathbf{z}\neq 0$ and with different values of $\mathbf{y}$.
Most of these solutions will provide upper bounds that are actually not very tight. It is thus important to select \textit{suitable} $\mathbf{y}$ and $\mathbf{z}$.
We propose four ways to select $\mathbf{y}$ and $\mathbf{z}$. Two of these ways are based on the sequential optimisation of two objectives. The third and fourth introduce an additional constraint.
\begin{itemize}
    \item The first method (called the \textit{left} linear robust solution 
    or \textit{robust line left} in the experiments) requires two optimisations. It first optimises $\mathcal{P}^{AR}_{\lambda_1, \lambda_2} (\lambda_1)$, with optimal value $v$. It then solves $\mathcal{P}^{AR}_{\lambda_1, \lambda_2} (\lambda_2)$ with the additional constraint $\mathbf{c}^t(\mathbf{y}+\lambda_1\mathbf{z}) = v$ whose optimal solution is $\mathbf y^l,\mathbf z^l$. The bound 
    $h^{AR}_{\mathbf y^l,\mathbf z^l}(\lambda)$
     is the tightest around $\lambda_1$ and among those, the one that has the lowest possible slope. 
    \item Similarly, the second method works in the same manner but in reverse order, and is called the \textit{right} linear robust solution 
    or \textit{robust line right} in the experiments. It first optimises $\mathcal{P}^{AR}_{\lambda_1, \lambda_2} (\lambda_2)$ with optimal value $w$,  then solves $\mathcal{P}^{AR}_{\lambda_1, \lambda_2} (\lambda_1)$ with the additional constraint $\mathbf{c}^t(\mathbf{y}+\lambda_2\mathbf{z}) = w$.
    \item The third method adds the constraint $\mathbf{c}^t \mathbf{z}= \delta$ to the problem $\mathcal{P}^{AR}_{\lambda_1, \lambda_2} (\lambda_1)$ to force the variable slope to be $\delta$ and optimises the objective $\mathbf{c}^t \mathbf{y}$ (as $\mathbf{c}^t \mathbf{z}$ is now constant). In the experiments, we fix the slopes to two values. The first one is $\delta = \frac{f(\lambda_2)-f(\lambda_1)}{\lambda_2-\lambda_1}$, called robust fixed slope pairwise. The second is $\delta = 0$, called robust yzflat. However, any $\delta$ can be chosen. Note that in the case $\delta=0$, the solution $\mathbf{y}+\lambda \mathbf{z}$ is not necessarily ``flat'' in the solution space (i.e. $\mathbf{z}$ may not be $\mathbf{0}$) as illustrated in the Example \ref{ex:non_flat}.
\end{itemize}

\begin{example}\label{ex:non_flat}
    The problem presented in eq. \eqref{eq:example-robust-1} has no constant robust solution over the range $[-2,2]$ (see Example \ref{example:robust-1}). It has, however, variable robust solutions (even a flat one), as shown in Figure \ref{fig:example-robust-2}. 
    \begin{figure}[H]
        \centering        \includegraphics[width=0.65\textwidth]{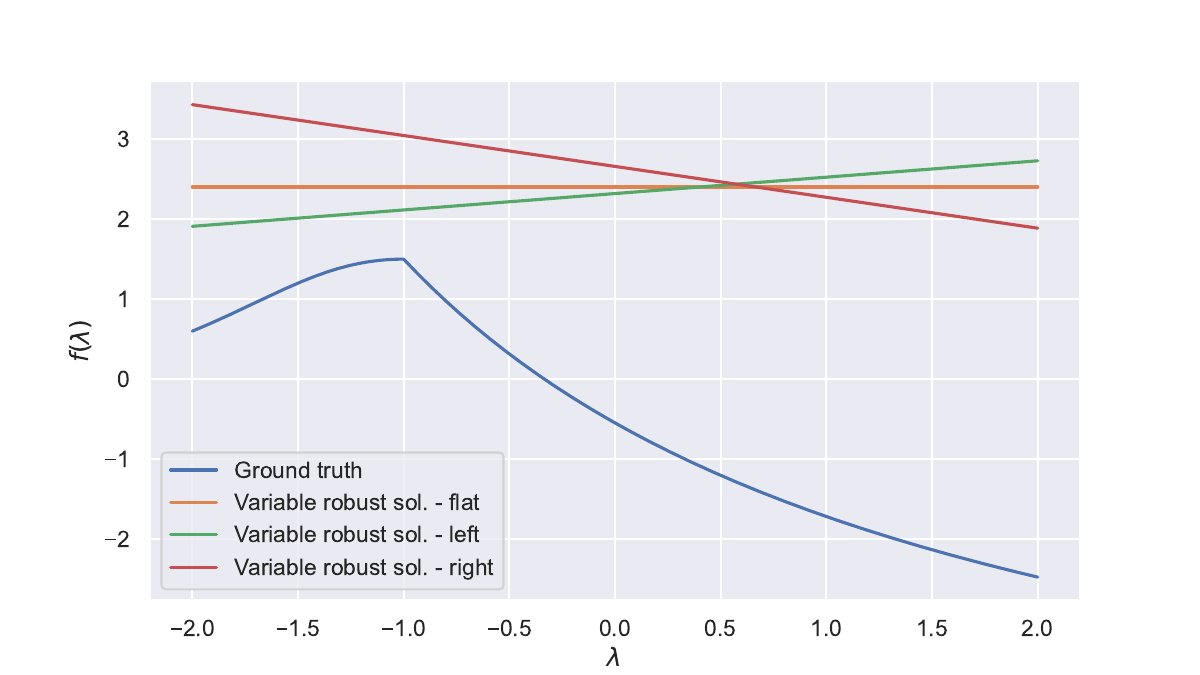}
        \caption{Bounds obtained with the different robust variable solution for the problem \eqref{eq:example-robust-1}.}
        \label{fig:example-robust-2}
    \end{figure}
\end{example}

\subsubsection*{Relation with the constant robust solution}

Example \ref{ex:non_flat} above shows that even when the bound is selected to be flat ($\mathbf{c}^t\mathbf{z}=0)$, it can exists even when the constant robust solution (as presented in Section \ref{sec:constant}) does not. Moreover, any constant robust solution maps to a variable one as shown in Theorem \ref{th:maprobustaffine}.

As a corollary, this implies that the bounds provided by the variable solutions are tighter (and exist more often) than the constant ones:

\begin{corollary}
    The (flat) variable robust solution $\mathbf{y}+\lambda\mathbf{z}$ respecting eq. \eqref{eq:affinespace} with $\mathbf{c}^t\mathbf{z}=0$ and minimizing $\mathbf{c}^t(\mathbf{y}+\lambda\mathbf{z}) = \mathbf{c}^t\mathbf{y}$ provides a tighter (or equal) upper bound than the constant robust solution $\mathbf{x}$ minimizing $\mathbf{c}^t\mathbf{x}$ if both exist:
    \begin{align}
        \mathbf{c}^t(\mathbf{y}+\lambda\mathbf{z}) = \mathbf{c}^t\mathbf{y} \leq \mathbf{c}^t\mathbf{x}.
    \end{align}
\end{corollary}

\subsection{Envelope of robust linear solutions}\label{sec:robust_concav_env}
In the previous section, we proved that, on linear problems, we can find a robust variable solution with several possible slopes. In this section, we introduce a bound tightness criterion to find the interval $[\nu_1, \nu_2]$ upon which a given optimal (for a certain value of $\lambda=\nu$) robust variable solution $\mathbf{y}+\lambda\mathbf{z}$ remains optimal over the interval, i.e. on that particular interval, we cannot find another robust variable solution whose objective function improves. From that criterion, coupled with warm starting, we describe a corresponding bound and discuss its strengths and weaknesses. 

Consider eq. \eqref{eq:h} with $\mu=\lambda$, $\mathcal{P}^{AR}_{\lambda_1, \lambda_2} (\lambda)$. An interesting observation is that $\lambda$ is now only present in the objective function; we have reduced a generic problem where modifications are applied to the constraints' coefficients to a simpler problem where the modifications are only located on the objective function coefficients. In principle, we would have to solve this optimization problem for every value of $\lambda$. In the following, we show how to do it without reoptimising for an infinite number of values of $\lambda$.

Assume that we find for a specific value $\nu\in[\lambda_1,\lambda_2]$, an optimal solution $(\mathbf y_\nu,\mathbf z_\nu)$ leading to a bound $h^{AR}_{\mathbf y_\nu,\mathbf z_\nu}(\lambda)$.
Techniques based on reduced costs can be used to find the interval $[\nu_1,\nu_2]$ where $(\mathbf y_\nu,\mathbf z_\nu)$ remains optimal for all problems $\mathcal P^{AR}_{\lambda_1,\lambda_2}(\lambda)$ with $\lambda\in [\nu_1,\nu_2]$.

In order to do this, and to simplify part of the proofs below, we rewrite the problem $\mathcal{P}^{AR}_{\lambda_1, \lambda_2} (\lambda)$ by creating explicit slack variables and putting all the constraints into a single matrix $M$ and the right-hand side in a vector $\mathbf{E}$:

\begin{mini}
{\mathbf{y},\mathbf{z}}
{\begin{pmatrix}
        \mathbf{c}^t & \lambda \mathbf{c}^t & \mathbf{0}
    \end{pmatrix} \begin{pmatrix}
        \mathbf{y} \\ \mathbf{z} \\ \mathbf{s}
    \end{pmatrix}}
{}
{\mathcal{P}^{AR}_{\lambda_1, \lambda_2} (\lambda) \equiv }
\addConstraint{M \begin{pmatrix}
        \mathbf{y} \\ \mathbf{z} \\ \mathbf{s}
    \end{pmatrix}}{= \mathbf{E},}{}
    \addConstraint{\mathbf{s}}{\geq \mathbf{0}.}{}
\end{mini}

We use this formulation to derive the associated theorem.

\begin{theorem}[Robust bound tightness criterion] Given $\nu \in [\lambda_1, \lambda_2]$, let $\langle \mathbf{y}^*, \mathbf{z}^*, \mathbf{s}^* \rangle$ be an optimal basic solution of the problem $\mathcal{P}^{AR}_{\lambda_1, \lambda_2} (\nu)$ and let $B$ and $N$ be the basic and non-basic sets of variables associated to the solution. Additionally, let $\mathbf{r}$ be the vector of reduced costs. For any $\lambda$ such that
\begin{align}
    (\lambda-\nu)
    \left(\begin{pmatrix}
        \mathbf{0} & \mathbf{c}^t_{Nz} & \mathbf{0}
    \end{pmatrix} - 
    \begin{pmatrix}
        \mathbf{0} & \mathbf{c}^t_{Bz} & \mathbf{0}
    \end{pmatrix} M_B^{-1} M_N\right) \geq - \mathbf{r},
\end{align}
$\langle \mathbf{y}^*, \mathbf{z}^*, \mathbf{s}^* \rangle$ is an optimal solution to $\mathcal{P}^{AR}_{\lambda_1, \lambda_2} (\lambda)$, where $\mathbf{c}^t_{B_z}$ (resp. $\mathbf{c}^t_{N_z}$) is the vector $\mathbf{c}^t$ restricted to the variables in $B_z$ (resp. $N_z$), itself the subset of $\mathbf{z}$ variables in $B$ (resp. $N$). 
    \label{theorem:line-range}
\end{theorem}

\begin{proof}
    The reduced costs for the basis related to $\langle \mathbf{y}^*, \mathbf{z}^*, \mathbf{s}^* \rangle$ on the problem $\mathcal{P}^{AR}_{\lambda_1, \lambda_2} (\nu)$ are
    \begin{align}
        \mathbf{r} := \begin{pmatrix}
            \mathbf{c}^t_{N_y} & \nu \mathbf{c}^t_{N_z} & \mathbf{0}
        \end{pmatrix} - \begin{pmatrix}
            \mathbf{c}^t_{B_y} & \nu \mathbf{c}^t_{B_z} & \mathbf{0}
        \end{pmatrix} M_B^{-1} M_N
    \end{align}
    with $\mathbf{r} \geq \mathbf{0}$, as the solution is optimal. If we now consider the same basis on the slightly modified problem $\mathcal{P}^{AR}_{\lambda_1, \lambda_2}(\nu + \delta)$, we obtain that the new reduced costs are:
    \begin{align}
        &\begin{pmatrix}
            \mathbf{c}^t_{Ny} & (\nu+\delta)\mathbf{c}^t_{Nz} & \mathbf{0}
        \end{pmatrix} - \begin{pmatrix}
        \mathbf{c}^t_{By} & (\nu+\delta)\mathbf{c}^t_{Bz} & \mathbf{0}
        \end{pmatrix} M_B^{-1} M_N\\
        =\ &\mathbf{r} + \begin{pmatrix}
            \mathbf{0} & \delta \mathbf{c}^t_{Nz} & \mathbf{0}
        \end{pmatrix} - \begin{pmatrix}
            \mathbf{0} & \delta \mathbf{c}^t_{Bz} & \mathbf{0}
        \end{pmatrix} M_B^{-1} M_N
    \end{align}
    For the basis to stay optimal for the given $\delta$, these reduced costs must be nonnegative. Therefore, we have that
    \begin{align*}
         \begin{pmatrix}
            \mathbf{0} & \delta \mathbf{c}^t_{Nz} & \mathbf{0}
        \end{pmatrix} - \begin{pmatrix}
            \mathbf{0} & \delta \mathbf{c}^t_{Bz} & \mathbf{0}
        \end{pmatrix} M_B^{-1} M_N \geq -\mathbf{r}
    \end{align*}
    
    If we pose $\delta = \lambda - \nu$, we obtain the conditions stated initially. 
\end{proof}

\subsection*{Derived bounding method}

Theorem \ref{theorem:line-range} allows us to find the interval $[\nu_1,\nu_2]$ for which a given robust linear (basic) solution $(\mathbf y,\mathbf z)$ remains optimal for  $\mathcal P^{AR}_{\lambda_1,\lambda_2}(\lambda)$ for all  $\lambda\in[\nu_1,\nu_2]$.
Once the basic solution is known, computing the range of $[\nu_1,\nu_2]$ for which it stays optimal can be done in $\mathcal{O}(m^2 + mn)$ where $n$ is the number of variables in the (full) problem and $m$ the number of constraints. Most of the operations are indeed vector-matrix multiplications and rely on standard warm starting techniques.\\

As $\lambda$ only appears in the objective of $\mathcal{P}^{AR}_{\lambda_1, \lambda_2}(\lambda)$, this is the classical parametric simplex method for a linearly varying objective, introduced by Gass and Saaty \cite{obj_param1,gass1955computational}, applied to the reformulated problem $\mathcal{P}^{AR}_{\lambda_1, \lambda_2}$ rather than to $\mathcal{P}(\lambda)$ itself. This procedure paves the way to compute a so-called ``envelope'' of robust linear solutions for linear problems, i.e. an envelope containing the best combination of lower and upper linear robust bounds, in a fast way using simplex warm starting. For minimization problems, the upper bounds give together a concave envelope, while the lower bounds give a convex one. Their combination form a convex space. 

Given an interval $[\lambda_1, \lambda_2]$, the steps to compute this envelope are given as follows: 
\begin{itemize}
    \item First, we solve $\mathcal{P}^{AR}_{\lambda_1, \lambda_2} (\lambda_1)$ and retrieve an optimal basic solution $\langle \mathbf{y}_{\lambda_1}, \mathbf{z}_{\lambda_1}, \mathbf{s}_{\lambda_1}\rangle$with its associated basis.
    \item Second, we compute the largest value of $\delta$ for which the current solution is still optimal, i.e. by using Theorem \ref{theorem:line-range} we take the largest $\delta$ such that the reduced costs of  $\mathcal{P}^{AR}_{\lambda_1,\lambda_2}(\lambda_1+\delta)$ are nonnegative.
    \item Third, we use warm starting methods to reoptimise the problem $\mathcal{P}^{AR}_{\lambda_1, \lambda_2} (\lambda_1+\delta+\epsilon)$. We get a new optimal basis and new optimal solution.
    \item Fourth, we repeat the three previous steps until $\lambda_2$ is reached.
    \item  The algorithm finishes and produces a finite number of solutions (as the polytope of the problem has a finite number of vertices), which all taken together form a concave upper bound for the range. Applied to the dual problem, the same procedure gives a convex lower bound. 
\end{itemize}

In Figure \ref{fig:concav-robust}, we illustrate the algorithm on both upper and lower bounds. For the lower bound, the algorithm first finds the linear bound in purple that minimizes the problem for $\lambda=-4$ and its equivalent basis. Then iteratively, it updates the next basis, found by using the optimality criterion, and reoptimises for the several values of $\lambda$ obtaining the bound: pink, grey, yellow and finally light blue in that order. For each reoptimisation, the previous basis is used, upon it, a few simplex iterations are performed to find the new optimal objective and basis.\\

The main drawback of the method resides in the need for an optimal basis whereas the other methods only need an optimal solution. Similarly to the robust variable solution, this method reformulates the problem using $2n$ variables and $2m_1+3m_2$ constraints. Finding an optimal basis would require the use of either the simplex algorithm or interior points methods with crossover enabled. This operation may take a lot of time. However, if an optimal basis is known, the method is fast and accurate. In the experiments this bound is named the robust envelope.

\begin{figure}
\centering\includegraphics[width=0.65\textwidth]{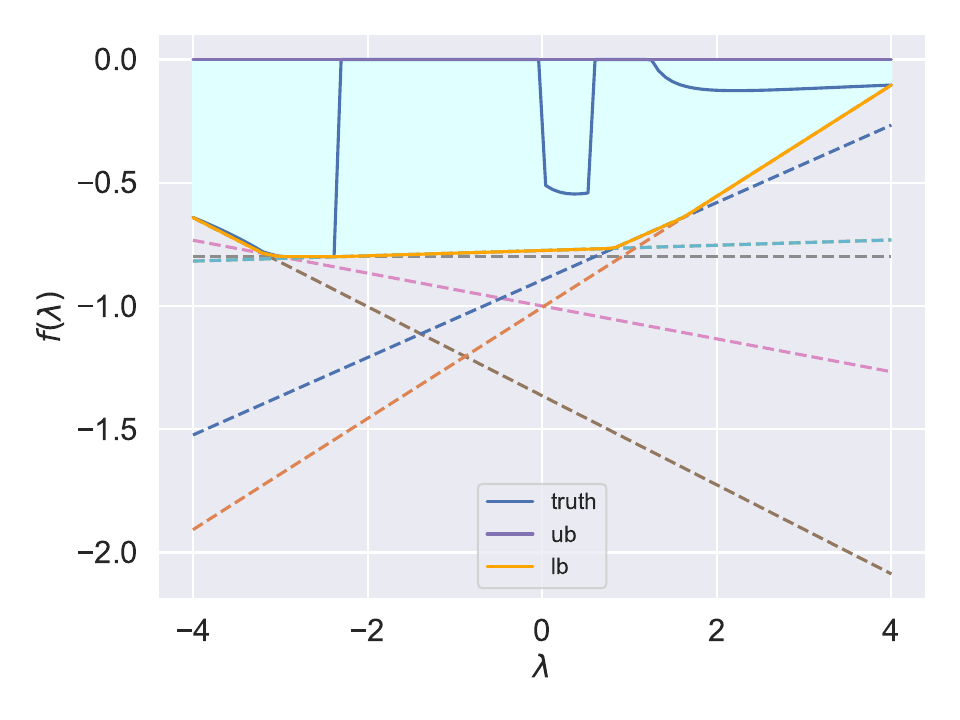}
    \caption{Illustration of the robust envelope algorithm on problem $\mathcal{P}_{\text{toy 1}}$ (see Annex \ref{annex:examples})}
    \label{fig:concav-robust}
\end{figure}

\subsection{Coefficient-wise bounding methods}\label{sec:coeff_wise}

In the previous methods, and especially in the case of the robust constant solution, we look for a robust feasible solution that is valid for the problem with infinitely many constraints.
The drawback of that method (and also of the affine robust solution) is computational. Indeed we double the number of constraints that are 
affected by $D$. The good point is that we keep a linear program, but that linear program is itself computationally more demanding than the
nominal linear program. In this section, we propose two bounding methods : one relaxation (providing a lower bound) and one restriction (providing an upper bound) that keep the size of the nominal linear program. The essential idea is to relax or restrict every coefficient that is potentially modified by the parameter $\lambda$. Throughout this section we need that additional assumption that the variables are all nonnegative.
Observe that we can transform free variables into the difference of two nonnegative variables if we want to apply the techniques presented in this section to problems with free variables. 

To make things clear, we consider the problem
\begin{align}
\begin{split}
    \mathcal{P}^{\geq 0}(\lambda)\equiv\min\quad &\mathbf{c}^t \mathbf{x}\\
    \text{ s.t } \quad&A_1\mathbf{x} \le \mathbf{b_1}\\
    & A_2\mathbf{x}+\lambda D\mathbf{x}\le \mathbf{b_2}\\
    & \mathbf{x} \in \mathbb{R}^{n_c}_+\times \mathbb{Z}^{n_i}_+.
\end{split}\label{eq:nonneg}
\end{align}

First, a tool to define the corresponding relaxation and restriction is defined.

\begin{definition}
Consider a linear function $a\lambda$ of the variable $\lambda$ and the fixed interval $[\lambda_1,\lambda_2]$. Let us define $(a\lambda)^{\uparrow}$ and $(a\lambda)^{\downarrow}$ as
\begin{align}
    (a\lambda)^{\uparrow} &= \max_{\lambda\in [\lambda_1,\lambda_2]} (a\lambda) = \begin{cases}
        a\lambda_1 & \text{ if } a \leq 0\\
        a\lambda_2 & \text{ otherwise}
    \end{cases} &
    (a\lambda)^{\downarrow} &= \min_{\lambda\in [\lambda_1,\lambda_2]} (a\lambda) = \begin{cases}
        a\lambda_1 & \text{ if } a \geq 0\\
        a\lambda_2 & \text{ otherwise}
    \end{cases}.
\end{align}
\end{definition}

\begin{theorem}
The problem
\begin{align}
\begin{split}
    \mathcal{P}^{\sqsubset}(\lambda)\equiv\min\quad &\mathbf{c}^t \mathbf{x}\\
    \text{ s.t } \quad&A_1\mathbf{x} \le \mathbf{b_1}\\
    & \sum_{j=1}^n a^{(2)}_{ij} x_j+\sum_{j=1}^n (\lambda d_{ij})^{\uparrow} x_j\le b_2^{i}\qquad \forall i\\
    & \mathbf{x} \in \mathbb{R}^{n_c}_+\times \mathbb{Z}^{n_i}_+
\end{split}\label{restric}
\end{align} 
is a restriction of $\mathcal{P}^{\geq 0}(\lambda)$ for all $\lambda \in [\lambda_1,\lambda_2]$. Hence, its optimal value is an upper bound of the optimal value of $\mathcal{P}^{\geq 0}(\lambda)$ for all $\lambda \in [\lambda_1,\lambda_2]$. Similarly, the problem
\begin{align}
\begin{split}
    \mathcal{P}^{\sqsupset}(\lambda)\equiv\min\quad &\mathbf{c}^t \mathbf{x}\\
    \text{ s.t } \quad&A_1\mathbf{x} \le \mathbf{b_1}\\
    & \sum_{j=1}^n a^{(2)}_{ij} x_j+\sum_{j=1}^n (\lambda d_{ij})^{\downarrow} x_j\le b_2^{i}\qquad \forall i\\
    & \mathbf{x} \in \mathbb{R}^{n_c}_+\times \mathbb{Z}^{n_i}_+
\end{split}\label{relax}
\end{align} 
is a relaxation of $\mathcal{P}^{\geq 0}(\lambda)$ for all $\lambda \in [\lambda_1,\lambda_2]$. Hence, its optimal value is a lower bound of the optimal value of $\mathcal{P}^{\geq 0}(\lambda)$ for all $\lambda \in [\lambda_1,\lambda_2]$.
\end{theorem}
\begin{proof}
 Let $\lambda \in [\lambda_1, \lambda_2]$. By definition, $(\lambda d_{ij})^{\downarrow} \leq \lambda d_{ij} \leq (\lambda d_{ij})^{\uparrow}$ for all $i, j$. Multiplying by $x_j \geq 0$ and summing over $j$ gives, for every row $i$ and every $\mathbf{x} \geq \mathbf{0}$,
 \begin{align*}
 \sum_{j=1}^n a^{(2)}_{ij} x_j+\sum_{j=1}^n (\lambda d_{ij})^{\downarrow} x_j \;\leq\; \sum_{j=1}^n (a^{(2)}_{ij} + \lambda d_{ij}) x_j \;\leq\; \sum_{j=1}^n a^{(2)}_{ij} x_j+\sum_{j=1}^n (\lambda d_{ij})^{\uparrow} x_j.
 \end{align*}
 Hence, any feasible solution of \eqref{restric} is feasible for $\mathcal{P}^{\geq 0}(\lambda)$, and any feasible solution of $\mathcal{P}^{\geq 0}(\lambda)$ is feasible for \eqref{relax}, the other constraints being identical.
\end{proof}

\begin{theorem}
    In terms of bound tightness, the restricted version of this bound is always more restrictive (less tight) than the robust flat bound.
\end{theorem}

 \begin{proof}
 The robust flat bound for nonnegative problems can be rewritten as
\begin{align*}
\min\quad &\mathbf{c}^t \mathbf{x}\\
\text{ s.t } \quad&A_1\mathbf{x} \le \mathbf{b_1}\\
& \sum_{j=1}^n a^{(2)}_{ij} x_j+\max_{\lambda\in[\lambda_1, \lambda_2]}(\sum_{j=1}^n \lambda d_{ij}x_j)\le b_2^{i}\qquad \forall i\\
    & \mathbf{x} \in \mathbb{R}^{n_c}_+\times \mathbb{Z}^{n_i}_+
\end{align*}
which is reformulated as 
\begin{align*}
\min\quad &\mathbf{c}^t \mathbf{x}\\
\text{ s.t } \quad&A_1\mathbf{x} \le \mathbf{b_1}\\
& \sum_{j=1}^n a^{(2)}_{ij} x_j+\sum_{j=1}^n \lambda_1 d_{ij}x_j\le b_2^{i}\qquad \forall i\\
& \sum_{j=1}^n a^{(2)}_{ij} x_j+\sum_{j=1}^n \lambda_2 d_{ij}x_j\le b_2^{i}\qquad \forall i\\
    & \mathbf{x} \in \mathbb{R}^{n_c}_+\times \mathbb{Z}^{n_i}_+.
\end{align*}
As we have that $\max_{\lambda\in[\lambda_1, \lambda_2]}(\sum_{j=1}^n \lambda d_{ij}x_j) \leq \sum_{j=1}^n \max_{\lambda\in[\lambda_1, \lambda_2]} \lambda d_{ij}x_j, \forall i$, the terms 
\begin{align*}
    \sum_{j=1}^n a^{(2)}_{ij} x_j+\sum_{j=1}^n (\lambda d_{ij})^{\uparrow} x_j &\geq \sum_{j=1}^n a^{(2)}_{ij} x_j+\sum_{j=1}^n \lambda_1 d_{ij}x_j \\ \text{and }\sum_{j=1}^n a^{(2)}_{ij} x_j+\sum_{j=1}^n (\lambda d_{ij})^{\uparrow} x_j &\geq \sum_{j=1}^n a^{(2)}_{ij} x_j+\sum_{j=1}^n \lambda_2 d_{ij}x_j, \forall i 
\end{align*}
making it harder for the constraint to be respected therefore being more restrictive.
 \end{proof}

However, in terms of computation, this bound keeps the same number of constraints and variables as the nominal problem, making it faster to compute than the robust flat bound.

These two bounds are closely related to linear programming with interval coefficients, where each coefficient can take any value in a given interval, independently of the other coefficients \cite{chinneck2000interval}. For problems with sign-restricted variables and inequality constraints, the best (resp. worst) optimal value over all the possible coefficient values is obtained by solving a single LP in which each coefficient is set to the end of its interval that yields the largest (resp. smallest) feasible region \cite[Algorithm 1]{chinneck2000interval}. Applied to the intervals $[a^{(2)}_{ij} + (\lambda d_{ij})^{\downarrow}, a^{(2)}_{ij} + (\lambda d_{ij})^{\uparrow}]$ spanned by the coefficients when $\lambda$ varies in $[\lambda_1, \lambda_2]$, these two LPs are exactly the relaxation \eqref{relax} and the restriction \eqref{restric}. These intervals ignore that all the coefficients vary with the same $\lambda$. This is what keeps the bounds cheap, at the price of tightness on large intervals; the adaptive refinement algorithm of Section~\ref{sec:it_algo} reduces this loss by shrinking the intervals.

\subsection{Lower bounds using Lagrangian relaxations}\label{sec:lag_flat}

In this section, we propose two lower bounds based on Lagrangian relaxations of \eqref{eq:separated}. 
Let us fix $\bm{\rho}\in \mathbb R^{m_2}_-$ and define the problem $\mathcal{P}^{L}(\bm{\rho}, \lambda)$ as the Lagrangian relaxation of $\mathcal{P}(\lambda)$, obtained by dualizing the second set of constraints, with dual multipliers $\bm\rho\leq 0$:
    \begin{mini}
    {x}
    {\mathbf{c}^t\mathbf{x} - \bm{\rho}^t ((A_2 + \lambda D) \mathbf{x} - \mathbf{b_2})}
    {}
    {\mathcal{P}^{L}(\bm{\rho}, \lambda) \equiv }
    \addConstraint{A_1 \mathbf{x}}{\leq \mathbf{b_1}}{}
    \addConstraint{\mathbf{x}}{\in \mathbb{R}^{n_c} \times \mathbb{Z}^{n_i}.}\label{lagrange}
    \end{mini}
The optimal objective of $\mathcal{P}^{L}(\bm{\rho}, \lambda)$ is denoted by $h^L(\bm{\rho}, \lambda)$ and provides a lower bound to $f(\lambda)$ for every $\lambda$. For a fixed $\lambda^*$ and a linear problem ($n_i = 0$), maximizing $h^L(\bm{\rho}, \lambda^*)$ over $\bm\rho \leq 0$ is simply the LP dual of $\mathcal{P}(\lambda^*)$. Our use of the Lagrangian relaxation is different: we consider the whole family of problems $\mathcal{P}(\lambda)$, $\lambda \in [\lambda_1, \lambda_2]$, in which the perturbed constraints $(A_2 + \lambda D)\mathbf{x} \leq \mathbf{b_2}$ are bilinear in $(\mathbf{x}, \lambda)$. Dualizing exactly these bilinear constraints with a fixed multiplier $\bm\rho$ removes $\lambda$ from the constraints: in \eqref{lagrange}, $\lambda$ only appears, linearly, in the objective. This is thus not a full dualization of $\mathcal{P}(\lambda)$: only the perturbed constraints are dualized, while the other constraints $A_1\mathbf{x} \leq \mathbf{b_1}$ (and the integrality constraints, if any) are kept in \eqref{lagrange}, whose optimal solution therefore still adapts to $\lambda$. This allows bounding $f$ over the whole interval with a single $\bm\rho$, as shown below. An interesting property of $h^L(\bm{\rho}, \lambda)$ is that it is concave both in $\bm\rho$ for a fixed $\lambda$ and in $\lambda$ for a fixed $\bm\rho$. 

In practice, it is complicated to determine the right Lagrange multipliers   $\bm\rho$ in order to obtain the tightest possible lower bounds.
Our first suggestion is to use as Lagrange multipliers the values that are obtained for the dual problem for some specific values of $\lambda$.
For example, if the initial problem does not include integer variables, we can solve the linear problem $\mathcal P(\lambda_1)$ which typically comes with optimal dual variables from which we retrieve
the dual variables associated with the constraints dualized in the Lagrangian relaxation, we denote them by $\bm\rho^*_{\lambda_1}.$
Similarly, we may assume that we compute  $\bm\rho^*_{\lambda_2}.$
To ease notation, we need the following definition.
\begin{definition}
Consider two pairs of real numbers $(x_1,y_1) $ and $(x_2,y_2)$, we denote the linear interpolation between these two pairs of points as
\begin{align*}
    \ell[(x_1,x_2),(y_1,y_2)] (t) =  y_1 + (t-x_1) \frac{y_2-y_1}{x_2-x_1}.
\end{align*}

\end{definition}
In order to derive the first bounding method using the Lagrangian relaxation \eqref{lagrange}, we can use a specific value of $\bm \rho$ through the following result.
\begin{lemma}
    Consider $\bm \rho\leq 0$ fixed. 
    For every $\lambda \in [\lambda_1,\lambda_2]$, a lower bound to \eqref{lagrange} is given by 
    \begin{align}
    \ell[(\lambda_1,\lambda_2),(h^L(\bm \rho,\lambda_1),h^L(\bm \rho,\lambda_2))](\lambda)\leq  h^L(\bm \rho,\lambda) \leq f(\lambda).    
    \end{align}
    \label{lemmaLagGen}
\end{lemma}
\begin{proof}
    
    For a fixed $\bm\rho$, the feasible set of \eqref{lagrange} does not depend on $\lambda$ and, for every feasible $\mathbf{x}$, the objective of \eqref{lagrange} is an affine function of $\lambda$. Hence, $h^L(\bm \rho,\lambda)$ is the optimal value of a problem whose feasible set is fixed and whose objective varies linearly with $\lambda$: it is therefore concave in $\lambda$ (see e.g. \cite{bertsimas}, Theorem 5.3, for the continuous case; the same argument applies to mixed-integer variables). The result follows, as any line segment joining two points of the graph of a concave function lies below this graph. If $h^L(\bm \rho,\lambda_1)$ or $h^L(\bm \rho,\lambda_2)$ is $-\infty$, the bound is trivially valid.
\end{proof}
We can now use Lemma \ref{lemmaLagGen}, with specific values of $\bm \rho$, in the case with no integer variables.
\begin{theorem}
Assume that the initial problem \eqref{eq:separated} has no integer variables, i.e. $n_i=0.$
Let $\bm \rho^*_{\lambda_1}$ (resp. $\bm \rho^*_{\lambda_2}$) be the optimal dual variables related to constraints $A_2\bm x+\lambda_1D\bm x\leq \bm b_2$ (resp.  $A_2\bm x+\lambda_2D\bm x\leq \bm b_2$) in \eqref{eq:separated}
    when the initial problem \eqref{eq:separated} is solved optimally for $\lambda=\lambda_1$ (resp. $\lambda=\lambda_2$).
    We have, for all $\lambda \in [\lambda_1, \lambda_2]$,
    \begin{align}
        f(\lambda) &\geq  \ell[(\lambda_1,\lambda_2),(f(\lambda_1),h^L(\bm \rho^*_{\lambda_1},\lambda_2))](\lambda)\label{line1LAG}\\
        f(\lambda) &\geq  \ell[(\lambda_1,\lambda_2),(h^L(\bm \rho^*_{\lambda_2},\lambda_1),f(\lambda_2))](\lambda)\label{line2LAG}
    \end{align}
\end{theorem}
\begin{proof}
This follows directly from Lemma \ref{lemmaLagGen} by using $\bm \rho^*_{\lambda_1}$ (resp. $\bm \rho^*_{\lambda_2}$) and the fact that
$h^L(\bm \rho^*_{\lambda_1},\lambda_1) = f(\lambda_1)$ (resp. $h^L(\bm \rho^*_{\lambda_2},\lambda_2) = f(\lambda_2)$) by strong duality of linear programming.
\end{proof}
Both lines on the right-hand-sides of \eqref{line1LAG} and \eqref{line2LAG} provide a lower bound on $f(\lambda)$. Combining the two by taking their maximum provides in general an even tighter bound. This is what we call the bisegment Lagrangian bound.
\begin{theorem}
    Assume that the initial problem \eqref{eq:separated} has no integer variables, i.e. $n_i=0$. We have, for all $\lambda \in [\lambda_1, \lambda_2]$,
    \begin{align*}
        f(\lambda) \geq \max\{\ell[(\lambda_1,\lambda_2),(f(\lambda_1),h^L(\bm \rho^*_{\lambda_1},\lambda_2))](\lambda),\ell[(\lambda_1,\lambda_2),(h^L(\bm \rho^*_{\lambda_2},\lambda_1),f(\lambda_2))](\lambda)\}.
    \end{align*}
\end{theorem}
If we do have integer variables, Lemma \ref{lemmaLagGen} can also be used, using the optimal dual variables of the linear relaxation, for example.
In that case, we have no guarantee of the optimality of the Lagrangian relaxation for one side of the interval though. \medskip

The potential weakness of Lagrangian relaxation is that, by relaxing the full set of constraints affected by $\lambda$, we may end up with
an optimization problem that is too weakly constrained, often ending up with weak relaxations. Observe that a second bounding method can be obtained by combining the technique with Lagrangian relaxation with coefficient-wise relaxation in order to overcome this issue. For any $\bm \rho\leq 0$, we can therefore consider

\begin{mini}
    {x}
    {\mathbf{c}^t\mathbf{x} - \bm{\rho}^t ((A_2 + \lambda D) \mathbf{x} - \mathbf{b_2})}
    {}
    {\mathcal{P}^{L,\sqsupset}(\bm{\rho}, \lambda) \equiv }
    \addConstraint{A_1 \mathbf{x}}{\leq \mathbf{b_1}}{}
    \addConstraint{\sum_{j=1}^n a^{(2)}_{ij} x_j + \sum_{j=1}^n (\lambda d_{ij})^{\downarrow}x_j\leq b_2^{i} \quad \forall i}{}
    \addConstraint{\mathbf{x}\in \mathbb{R}^{n_c} \times \mathbb{Z}^{n_i}.}\label{lagrangeCW}
    \end{mini}
We denote by $h^{L,\sqsupset}(\bm \rho,\lambda)$ the optimal value of \eqref{lagrangeCW}. It is straightforward to see that 
$h^{L,\sqsupset}(\bm \rho,\lambda)$ is a lower bound on $f(\lambda)$ for every $\lambda$. All the results previously derived for the Lagrangian relaxation extend to the relaxation enforced by the coefficient-wise relaxation. The two bounds are illustrated in Figure \ref{fig:illustration_lag}.

\begin{figure}[H]%
    \centering
    \subfloat[\centering Lagrangian bisegment]{{\includegraphics[width=0.4\textwidth]{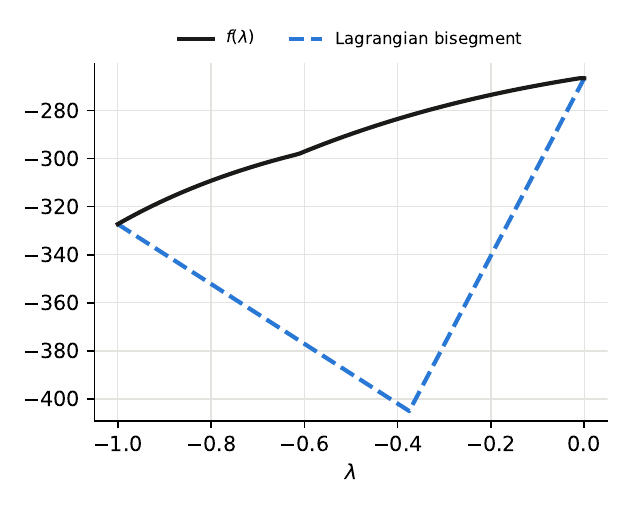} }}%
    \qquad
    \subfloat[\centering Lagrangian bisegment coefficient-wise]{{\includegraphics[width=0.4\textwidth]{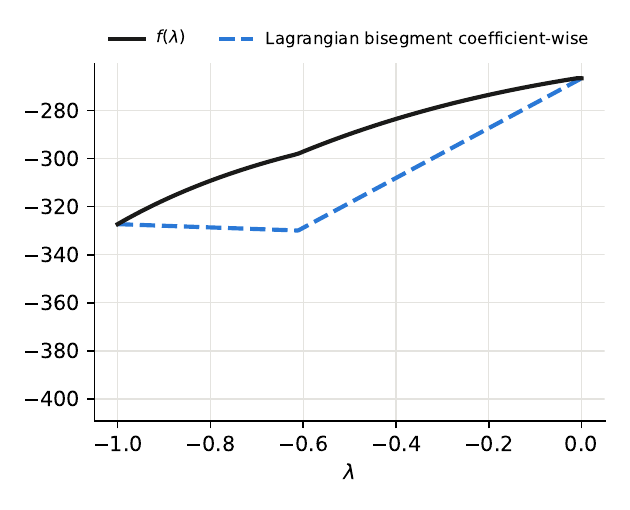} }}%
    \caption{Illustration of the Lagrangian bounds on the synthetic instance derived from the Netlib problem recipe, in which only equality rows are perturbed (Section~\ref{sec:dataset}), for a single bound on the interval $[-1, 0]$.}%
    \label{fig:illustration_lag}%
\end{figure}

\begin{figure}[h!]
    \centering
    \includegraphics[width=\textwidth]{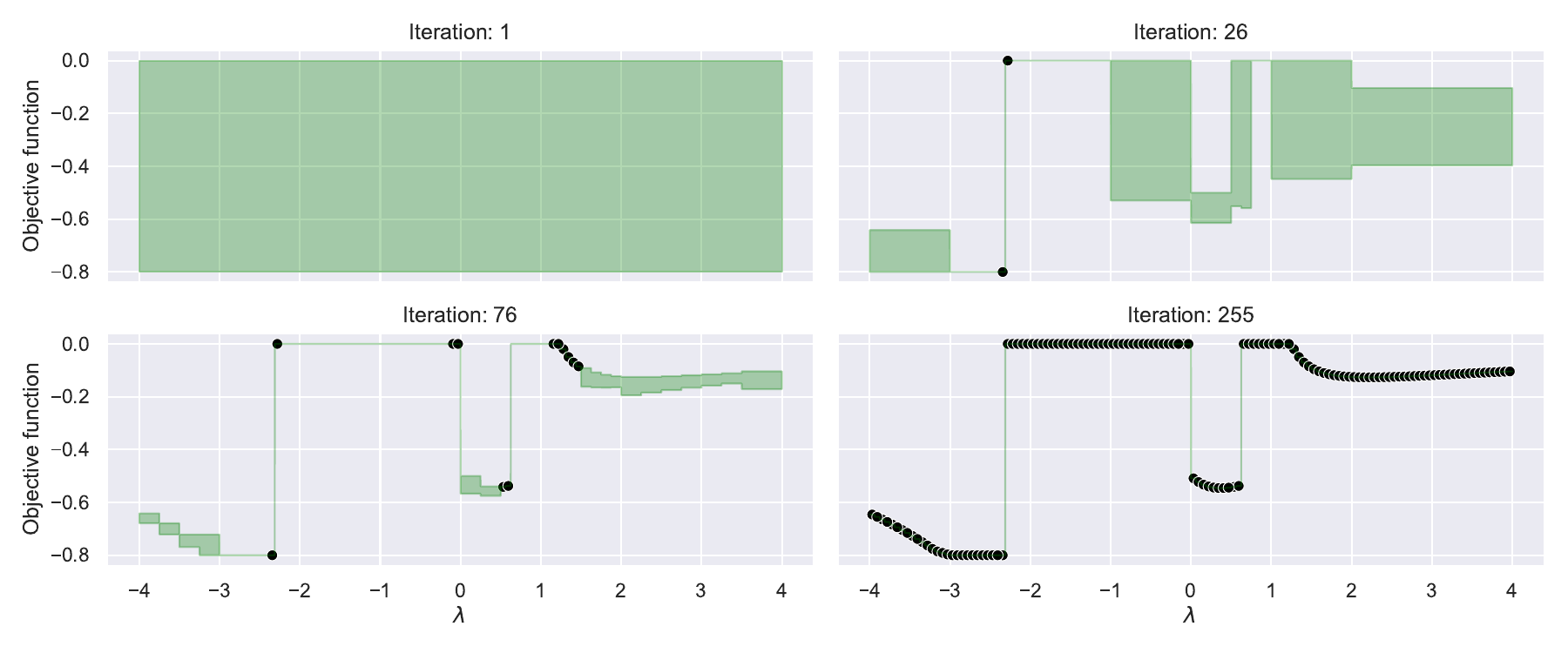}
    \caption{Four steps of the adaptive refinement algorithm using robust yzflat for both upper and lower bounds on Problem \eqref{ex:toy_1} (Annex \ref{annex:examples}). At any iteration, the algorithms ensures that $f(\lambda)$ lies in the green zone or is equal to the black dots.}
    \label{fig:toy_1_it_alg}
\end{figure}

\section{Adaptive refinement algorithm}\label{sec:it_algo}

The idea behind the adaptive refinement algorithm is simple: we have two sets of methods, one that generates an upper bound and one a lower bound. The resulting bounds can be refined by reducing the range $[\lambda_1, \lambda_2]$, i.e. dividing the range and reapplying the methods to find new upper and lower bounds for this new smaller range.

Algorithm \ref{alg:iterative} works using this principle. The node strategy chosen maintains a priority queue of ranges to explore, starting with $[\lambda_1, \lambda_2]$. It iteratively selects the range with the largest error (defined by the area between the bounds computed in each range, i.e. the integral over the range of the difference between the upper and the lower bound) and refines it, by dividing it into two equal parts and computing the bounds on these two parts, then reinserting them into the priority queue. Degenerate cases exist where the upper/lower bounds cannot be found for a given range (for example when the Lagrangian relaxation is unbounded); in such cases, we write that the lower/upper bounds are respectively $-\infty$ and $\infty$ and that the distance between the bounds is infinite. This can lead the algorithm to exclusively focus on these ranges. In order to avoid this behaviour, we add a stopping condition on the algorithm: when the range length being refined is smaller than a user-defined $\epsilon_\lambda$, the algorithm computes the ground truth $f(\lambda)$ at the middle of the range and stops refining this particular range (i.e. it considers that the error in this range is 0). This is an anytime algorithm: at any point in time, the algorithm maintains the sets of bounds that have been computed. Figure \ref{fig:toy_1_it_alg} shows four different steps of the adaptive refinement algorithm using robust yzflat for both upper and lower bounds, nominally the $1^{st}$, $26^{th}$, $76^{th}$ and last ($255^{th}$) iteration for Problem \eqref{ex:toy_1} (Toy 1 from Annex \ref{annex:examples}). The black dots are the true $f(\lambda)$ points computed. 

\begin{algorithm}[ht]
\begin{algorithmic}
\Function{it\_alg}{$\mathcal{P}$, $\underline{f}$, $\overline{f}$, $[\lambda_1, \lambda_2], \epsilon_\lambda, T$}
\State \textbf{Input: } $\mathcal{P}$ the problem
\State \phantom{\textbf{Input: }} $\underline{f}: \mathbb{R}^2 \mapsto (\mathbb{R} \mapsto \mathbb{R} \cup \{-\infty\})$ which provides lower bounds
\State \phantom{\textbf{Input: }} $\overline{f}: \mathbb{R}^2 \mapsto (\mathbb{R} \mapsto \mathbb{R} \cup \{\infty\})$ which provides upper bounds
\State \phantom{\textbf{Input: }} $[\lambda_1, \lambda_2]$ the range to consider 
\State \phantom{\textbf{Input: }} $\epsilon_\lambda$ the size of smallest interval to consider before computing a point 
\State \phantom{\textbf{Input: }} $T$ the time limit.
\State \textbf{Output: } $\mathcal{L},\mathcal{U},\mathcal{D}$ the sets of lower bounds, upper bounds and points (resp.) that have been computed.

\\
\State $\mathcal{T} \gets$ empty max-Priority Queue
    \Comment{Todo-list of ranges}
\State $\mathcal{L} \gets \emptyset$
    \Comment{Set of computed lower bounds}
\State $\mathcal{U} \gets \emptyset$
    \Comment{Set of computed upper bounds}
\State $\mathcal{D} \gets \emptyset$
    \Comment{Set of computed points}
\\
\State $lb_i, ub_i \gets \underline{f}(\lambda_1, \lambda_2), \overline{f}(\lambda_1, \lambda_2)$
\State Insert $[\lambda_1, \lambda_2]$ with priority $\int_{\lambda_1}^{\lambda_2} (ub_i(\lambda) - lb_i(\lambda))\, d\lambda$ in $\mathcal{T}$
\State $\mathcal{L} \gets \mathcal{L} \cup \{lb_i\}$
\State $\mathcal{U} \gets \mathcal{U} \cup \{ub_i\}$
\\
\While {$\mathcal{T} \neq \emptyset$ and time limit is not reached}
    \State $[\underline{\lambda_i}, \overline{\lambda_i}] \gets $ pop range from $\mathcal{T}$ with highest priority
    \State $\text{mid} \gets \frac{\overline{\lambda_i}+\underline{\lambda_i}}{2}$ 
    \If{$\overline{\lambda_i} - \underline{\lambda_i} > \epsilon_\lambda$}
        \State $lb_1, ub_1 \gets \underline{f}(\underline{\lambda_i}, \text{mid}), \overline{f}(\underline{\lambda_i}, \text{mid})$
        \State $lb_2, ub_2 \gets \underline{f}(\text{mid}, \overline{\lambda_i}), \overline{f}(\text{mid}, \overline{\lambda_i})$
        \State Insert $[\underline{\lambda_i}, \text{mid}]$ with priority $\int_{\underline{\lambda_i}}^{\text{mid}} (ub_1(\lambda) - lb_1(\lambda))\, d\lambda$ in $\mathcal{T}$
        \State Insert $[\text{mid}, \overline{\lambda_i}]$ with priority $\int_{\text{mid}}^{\overline{\lambda_i}} (ub_2(\lambda) - lb_2(\lambda))\, d\lambda$ in $\mathcal{T}$
        \State $\mathcal{L} \gets \mathcal{L} \cup \{lb_1, lb_2\}$
        \State $\mathcal{U} \gets \mathcal{U} \cup \{ub_1, ub_2\}$
    \Else \Comment{stop condition}
        \State $\mathcal{D} \gets \mathcal{D} \cup \{(\text{mid}, f_{\mathcal{P}}(\text{mid}))\}$
    \EndIf
\EndWhile
\State \Return $\mathcal{L}, \mathcal{U}, \mathcal{D}$
\EndFunction
\end{algorithmic}
\caption{Adaptive refinement algorithm}\label{alg:iterative}
\end{algorithm}

%% file: TexFiles/experiments.tex
\section{Experiments}\label{sec:expe}
In this section, we discuss the different bounding methods proposed and provide several empirical studies related to them. A summary of the nine different bounds method used in the experiments is given in Table \ref{tab:all_bounds}. The section is divided in four parts: 
\begin{itemize}
    \item First, we describe the datasets.
    \item Second, we define our three metrics: the availability, error and timing. We then compare the performance of each bounding method in terms of these metrics.
    \item Third, we study the influence of the rank of the perturbation on the bounds.
    \item Fourth, we discuss the performance of the proposed adaptive refinement algorithm for a combination of lower and upper bounds.
\end{itemize}

\begin{table}[H]
    \centering
    
    \resizebox{\textwidth}{!}{%
    \begin{tabular}{lllll}
         Approach  & Type & Section & Name & Comment\\
         \toprule
         Robust  & Constant & \ref{sec:constant} & Robust flat &\\
         
         Optimization& Linear in $\lambda$ &\ref{sec:variable}& Robust line left&  First optimize for $\lambda_1$ then for $\lambda_2$.\\
         
         &  && Robust line right&  First optimize for $\lambda_2$ then for $\lambda_1$.\\
         
         &  && Robust yzflat &  Add constraint $\mathbf{c^t}\mathbf{z} = 0$.\\
         
         & && Robust fixed slope pairwise & $\mathbf{c^t}\mathbf{z}=\delta$ where $\delta$ is the slope \\
         
         & & & & between the optimum for $\lambda_1$ and $\lambda_2$.\\
         
         & Envelope & \ref{sec:robust_concav_env} & Robust envelope & \\
         
         \midrule
         
         Coefficient-wise & Constant & \ref{sec:coeff_wise} & Coefficient-wise & Requires  $\mathbf x\geq 0.$\\
         
         \midrule
         Lagrangian & Piecewise linear & \ref{sec:lag_flat} &Lagrangian bisegment&\\
         Relaxations& Piecewise linear & \ref{sec:lag_flat}& Lagrangian bisegment coeff& Add coefficient-wise constraints, requires $\mathbf x\geq0$. \\
         \bottomrule
    \end{tabular}
    }
    \caption{Summary of all the bounding methods used in the experiments}
    \label{tab:all_bounds}
\end{table}

\subsection{Dataset}\label{sec:dataset}

To the best of our knowledge, there is no benchmark of problems whose constraint matrix depends on a parameter. We therefore built two datasets, summarized in Tables~\ref{tab:dataset} and~\ref{tab:dataset_app}. In both of them, $\lambda \in [-1, 1]$ and only the constraint matrix depends on $\lambda$: the objective, the right-hand side and the variable bounds do not. The datasets and the code that generates them are available on Zenodo \cite{zenodo_code}.

\paragraph*{Synthetic dataset.} The synthetic instances are derived from the 91 LPs of the Netlib collection \cite{netlib} and from the 240 MILPs of the MIPLIB 2017 benchmark set \cite{miplib}. We keep the problems whose nominal version is solved within a deterministic budget of 5 (Netlib) or 10 (MIPLIB) Gurobi work units on a single thread, roughly a few seconds per solve. Only genuine constraint rows with at least two nonzero coefficients can be perturbed: variable bounds, as well as rows with a single nonzero coefficient, which are bounds in disguise, are never perturbed. Two sets of instances are constructed separately from each problem: one in which only equality rows are perturbed, and one in which only inequality rows are perturbed. For each problem and construction, $R = \min(100, \lfloor 0.5 \cdot \text{number of eligible rows}\rfloor)$ rows are drawn uniformly without replacement, and $\min(3, \text{nnz})$ nonzero coefficients are drawn without replacement in each of these rows. Each drawn coefficient $a_{ij}$ becomes $a_{ij}(1 + \lambda r_{ij})$, i.e. $d_{ij} = r_{ij} a_{ij}$, with $|r_{ij}|$ drawn uniformly in $[0.1, 0.5]$ and a random sign, so that it never vanishes nor changes sign for $\lambda \in [-1, 1]$. Only existing nonzero coefficients are thus perturbed, so that the sparsity pattern of the problem is preserved. The choice of a large number of rows with a few coefficients each mimics a parameter that appears in many constraints; we study the influence of these two numbers in Section~\ref{sec:bound_assess}. Each candidate instance is evaluated on 21 equally spaced values of $\lambda$ and kept only if every point is solved to optimality within the work budget, and if $f$ is neither flat (its relative variation $(\max f - \min f)/\max |f|$ over the grid must be at least $0.5\%$, also when the nominal point $\lambda = 0$ is left out) nor linear (maximum deviation from its least-squares line of at least $2\%$ of its range). Up to 10 candidates, with seeds derived from the instance name, are tried for each problem and construction. This yields 147 instances, 121 LPs and 26 MILPs, summarized in Table~\ref{tab:dataset}. A separate subset of 68 instances keeps candidates for which $f$ is infeasible on part of the grid (with at least $30\%$ of feasible points); it is only used to discuss the availability of the bounds.

\begin{table}[ht]
    \centering
    \begin{tabular}{llrrr}
         Dataset & Construction & Type & \#instances & Median rel. var. of $f$\\
         \toprule
         Synthetic (Netlib) & equality rows & LP & 68 & 0.172\\
          & inequality rows & LP & 53 & 0.043\\
         Synthetic (MIPLIB) & equality rows & MILP & 5 & 0.119\\
          & inequality rows & MILP & 21 & 0.050\\
         \midrule
         Synthetic, infeasible on part of the grid & & LP / MILP & 64 / 4 & \\
         \bottomrule
    \end{tabular}
    \caption{Synthetic dataset. The relative variation of $f$ is $(\max f - \min f)/\max|f|$ over 21 equally spaced values of $\lambda \in [-1, 1]$.}
    \label{tab:dataset}
\end{table}

\paragraph*{Application dataset.} The application dataset contains 126 MILP instances built from 41 models of seven application domains, taken from public libraries and papers and listed in Table~\ref{tab:dataset_app}. In each instance, $\lambda$ represents a physical or economic parameter that appears in several coefficients of the constraint matrix, spread over several rows and columns; for instance, the reactances of the most loaded lines of the German transmission grid in the Kirchhoff voltage law constraints of a network-constrained unit commitment problem. The acceptance criteria are the same as for the synthetic dataset, except that the budget is a median solve time per point between 0.1 and 60 seconds as computed by the Gurobi solver on a M1 Macbook Pro. The models, their sources and the meaning of each parameter are documented with the dataset.

\begin{table}[ht]
    \centering
    \resizebox{\textwidth}{!}{%
    \begin{tabular}{lllrrrr}
         Domain & Problem & Source & \#models & \#instances & \#variables & Median rel. var. of $f$\\
         \toprule
         Logistics & Capacitated hub location & \cite{ernst1999hub,ernst1996hub} & 10 & 30 & 8\,000--125\,000 & 0.040\\
         Chemical engineering & Batch process scheduling & \cite{kondili1993stn} & 4 & 20 & 495--669 & 0.159\\
         Location science & Capacitated $p$-median & \cite{feng2022spopt} & 6 & 18 & 3\,296 & 0.065\\
         Telecommunications & Network design (SNDlib) & \cite{orlowski2010sndlib} & 4 & 17 & 375--782 & 0.096\\
         Finance & Index tracking & \cite{beasley2003indtrack} & 4 & 16 & 270--642 & 0.269\\
         Power systems & Unit commitment & \cite{knueven2020uc} & 9 & 13 & 44\,544--305\,712 & 0.016\\
         Power systems & Network-constrained unit commitment & \cite{brown2018pypsa,medjroubi2017scigrid} & 4 & 12 & 13\,360 & 0.032\\
         \midrule
         Total & & & 41 & 126 & 270--305\,712 & 0.062\\
         \bottomrule
    \end{tabular}
    }
    \caption{Application dataset. The relative variation of $f$ is $(\max f - \min f)/\max|f|$ over 21 equally spaced values of $\lambda \in [-1, 1]$.}
    \label{tab:dataset_app}
\end{table}

\paragraph*{Normalization.} The optimal values of the instances have very different magnitudes. To compare them, we normalize each instance: the same affine transformation is applied to $f$ and to all its bounds, such that the minimum (resp. maximum) of $f$ over the evaluation points of Section~\ref{sec:bound_assess} is mapped to 1 (resp. 2). The errors reported below are therefore fractions of the range of variation of $f$ over $[\lambda_1, \lambda_2]$, and not of the value of $f$. The median relative variation of $f$ reported in Tables~\ref{tab:dataset} and~\ref{tab:dataset_app} gives the order of magnitude of this range compared with $|f|$.

\begin{figure}[h!]
    \centering
    \includegraphics[width=0.8\textwidth]{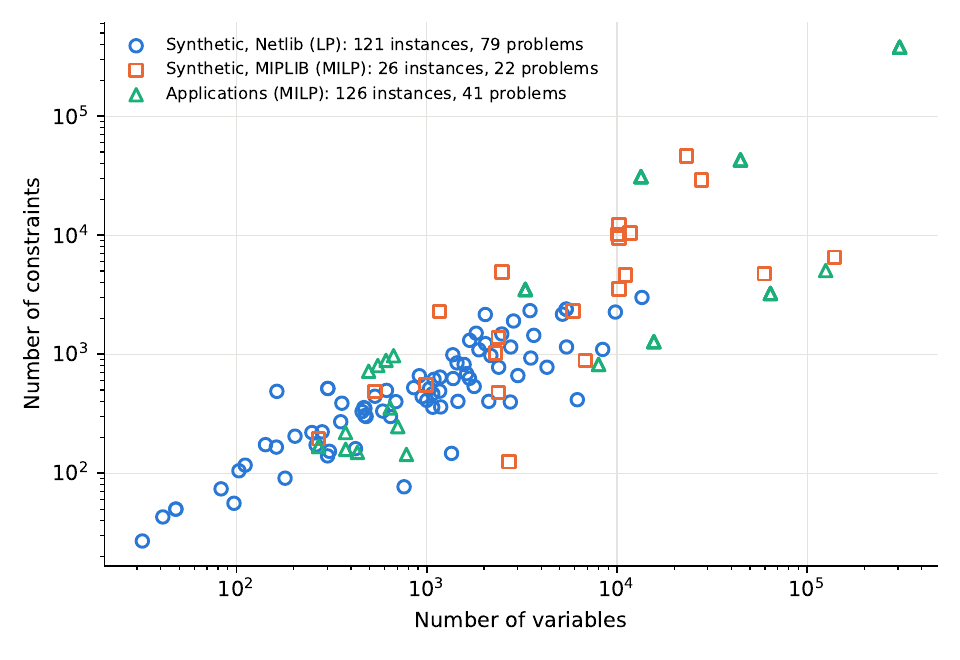}
    \caption{Number of constraints and of variables of the problems of the synthetic and application datasets (one marker per source problem or model; the instances built from the same problem have the same size).}
    \label{fig:dataset}
\end{figure}

\subsection{Bound assessment}\label{sec:bound_assess}
In this experiment, we apply each bounding method to each problem presented in Section \ref{sec:dataset}. We compare the generated bounds in terms of availability (percentage of the feasible points for which a bounding method returns a finite value), error and time taken to compute each bound in order to generate upper and lower bounds. We benchmark the performance of each bound based on these criteria. 

\begin{itemize}
    \item \textbf{Availability:} To assess the bounds, we first compute the exact optimal value $f(\lambda)$ at 100 equally spaced points of $[\lambda_1, \lambda_2]$, solving each problem from scratch. This set of $\lambda$ is defined as 
\begin{align*}
    \mathcal{S} &= \{\lambda_1+\frac{i}{99}(\lambda_2-\lambda_1) \mid i \in \{0, ..., 99\}\}, \qquad \mathcal{S}_F = \{\lambda \in \mathcal{S} \mid \mathcal{P}(\lambda) \text{ is feasible}\}.
\end{align*}
In the datasets of Section~\ref{sec:dataset}, $\mathcal{S}_F = \mathcal{S}$, except in the subset of instances with infeasible pieces. To define the availability of a bound, we defined the subset of feasible points for which a bound returns a finite value:
\begin{align}
    \mathcal{A}_b &= \{\lambda \in \mathcal{S}_F \mid \text{bound } b \text{ returns a finite value at }\lambda\}\ .
\end{align}
The \textit{availability} is defined as the percentage of points in $\mathcal{S}_F$ for which the bound is \textit{available}, i.e. returns a finite value:
\begin{align}
    \text{avail}_b = \frac{|\mathcal{A}_b|}{|\mathcal{S}_F|} \cdot 100\ .
\end{align}
Infeasible values of $\lambda$ are left out because $f(\lambda) = +\infty$ there, so that no upper bound can be finite. A bound is not available on a subinterval when its bounding problem has no finite optimal value. For instance, the constant robust bound requires a solution that is feasible for all $\lambda$ of the subinterval; there is none as soon as $\mathcal{P}(\lambda)$ is infeasible for one of them, and, with perturbed equality constraints, often as soon as the subinterval is not reduced to a point. An unavailable bound thus indicates a subinterval that is too large for this bound, not an ill-posed bounding problem: the availability increases when the interval is split into smaller subintervals.
\item \textbf{Error:} The root mean square error (RMSE) is defined as
\begin{align}
    \text{RMSE}_b = \sqrt{\frac{\sum_{\lambda\in \mathcal{A}_b}(\tilde b(\lambda) - \tilde f(\lambda))^2}{|\mathcal{A}_b|}}\ ,
\end{align}
where $\tilde b(\lambda)$ and $\tilde f(\lambda)$ are the bound value $b(\lambda)$ and the objective value $f(\lambda)$ after the normalization of Section~\ref{sec:dataset}, which maps the minimum (resp. maximum) of $f$ over $\mathcal{S}_F$ to 1 (resp. 2).  We choose to only compute the error for points where the bound exists. If the bound is not available at any point, the error is considered to be infinite. An error of $0$ means that the bound coincides with $f$ at all the points where it is available. \\
\item \textbf{Timing: } As the problems exhibit a wide variability in computation times, ranging from a few seconds to several minutes, let us consider the mean time $t_m$ required to compute $f(\lambda)$ for $\lambda \in \mathcal{S}$ (solving the problem from scratch) and the time to compute a bound $t_b$ on all the subintervals. The reported relative time is given by $t = \frac{t_b}{t_m}$. It is thus a number of evaluations of $f$. All the problems are solved with CPLEX 22.2 on a single thread, on the same machine, and each bound computation is limited to 5400 seconds.
\end{itemize}
 We split the interval $[\lambda_1, \lambda_2]$ uniformly in $N$ subintervals of same size, 
\begin{align}
    [\lambda_1+\frac{i}{N}(\lambda_2-\lambda_1), \lambda_1+\frac{i+1}{N}(\lambda_2-\lambda_1) ] & & \forall i \in \{0, ..., N-1\}\ .
\end{align}
We illustrate the impact of cutting the interval for $N=1$, $N=5$ and $N=10$ in Figure \ref{fig:toy_2_diff_N} for the robust envelope bound on one particular problem. We compare the bounding methods with $N=10$ subintervals in terms of error, timing and availability on the dataset and look for trends in terms of problem and modification type.
\begin{figure}[ht]
    \centering
    \includegraphics[width=\textwidth]{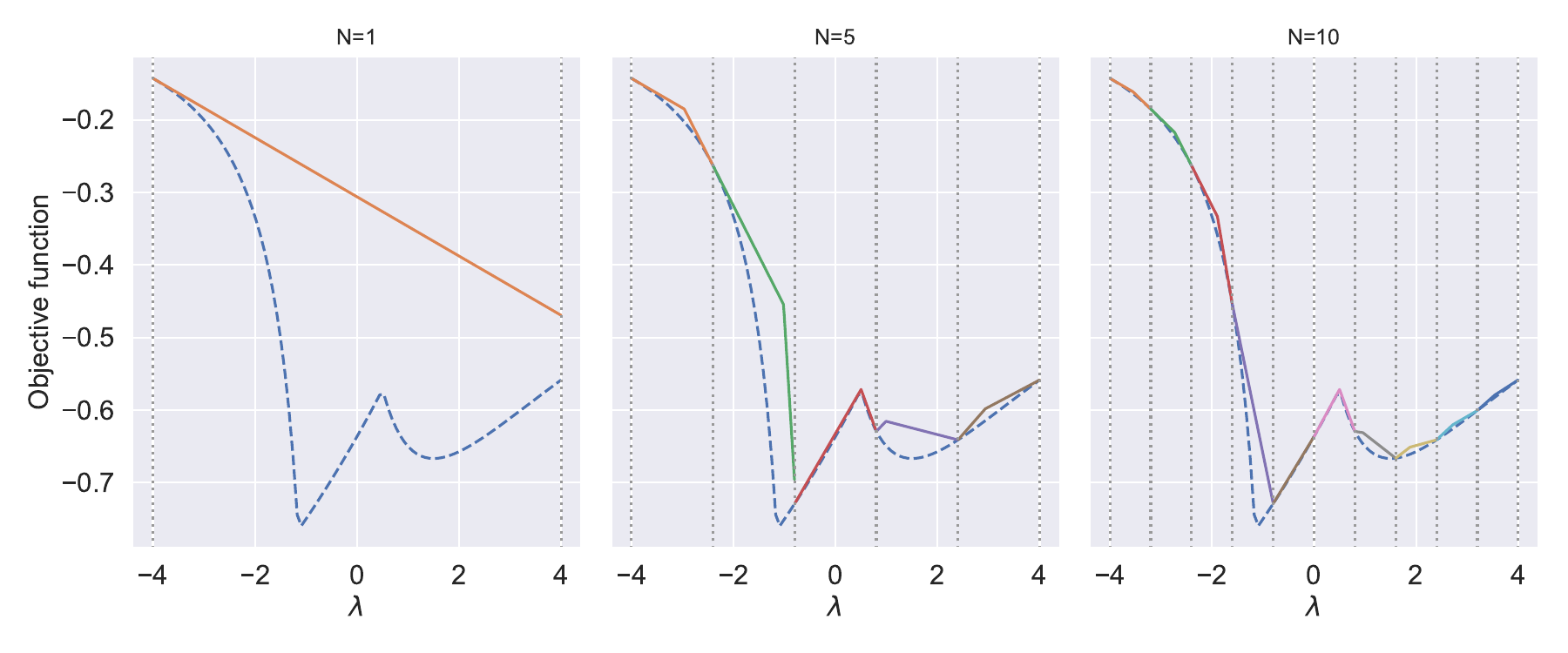}
    \caption{Illustration of robust envelope to provide an upper bound for the problem ``Toy 2" given in \eqref{ex:toy_2} in Annex \ref{annex:examples}.}
    \label{fig:toy_2_diff_N}
\end{figure}

\subsubsection*{Bounding methods comparison}

\begin{table}[ht]
    \centering
    \resizebox{\textwidth}{!}{%
\begin{tabular}{lrrrrrrrrrr}
\toprule
 & \multicolumn{4}{c}{Netlib (LP)} & \multicolumn{4}{c}{MIPLIB (MILP)} & \multicolumn{2}{c}{Applications} \\
 & \multicolumn{2}{c}{ineq (53)} & \multicolumn{2}{c}{eq (68)} & \multicolumn{2}{c}{ineq (21)} & \multicolumn{2}{c}{eq (5)} & \multicolumn{2}{c}{(MILP, 126)} \\
\cmidrule(lr){2-3}\cmidrule(lr){4-5}\cmidrule(lr){6-7}\cmidrule(lr){8-9}\cmidrule(lr){10-11}
 & lb & ub & lb & ub & lb & ub & lb & ub & lb & ub \\
\midrule
Robust flat & 71.7 & 98.5 & 72.1 & 25.0 &  & 99.5 &  & 40.0 &  & 87.3 \\
Robust yzflat & 98.7 & 99.8 & 95.7 & 64.4 &  &  &  &  &  &  \\
Robust line left & 83.8 & 97.2 & 81.0 & 59.4 &  &  &  &  &  &  \\
Robust line right & 84.3 & 93.8 & 81.9 & 58.5 &  &  &  &  &  &  \\
Robust fixed slope pairwise & 99.6 & 99.4 & 96.6 & 58.5 &  &  &  &  &  &  \\
Robust envelope & 99.8 & 100.0 & 98.5 & 64.4 &  &  &  &  &  &  \\
\midrule
Coefficient-wise & 96.2 & 96.2 & 98.5 & 25.0 & 100.0 & 99.5 & 100.0 & 40.0 & 99.8 & 88.9 \\
\midrule
Lagrangian bisegment & 62.8 & 27.5 & 36.2 & 8.2 & 91.4 &  & 80.0 &  & 93.7 &  \\
Lagrangian bisegment coef & 87.7 & 97.4 & 96.6 & 8.2 & 86.7 &  & 100.0 &  & 93.7 &  \\
\bottomrule
\end{tabular}
    }
    \caption{Mean availability of the bounds (in \%), on the feasible points $\mathcal{S}_F$, with $N=10$ subintervals. The number of instances of each dataset is given in parentheses. Empty cells indicate that the bounding method does not apply to this class of problems or type of bound.}
    \label{tab:avail}
\end{table}
\begin{table}[ht]
    \centering
    \resizebox{\textwidth}{!}{%
\begin{tabular}{lrrrrrrrrrr}
\toprule
 & \multicolumn{4}{c}{Netlib (LP)} & \multicolumn{4}{c}{MIPLIB (MILP)} & \multicolumn{2}{c}{Applications} \\
 & \multicolumn{2}{c}{ineq (53)} & \multicolumn{2}{c}{eq (68)} & \multicolumn{2}{c}{ineq (21)} & \multicolumn{2}{c}{eq (5)} & \multicolumn{2}{c}{(MILP, 126)} \\
\cmidrule(lr){2-3}\cmidrule(lr){4-5}\cmidrule(lr){6-7}\cmidrule(lr){8-9}\cmidrule(lr){10-11}
 & lb & ub & lb & ub & lb & ub & lb & ub & lb & ub \\
\midrule
Robust flat & 5.28 & 0.14 & 6.67 & $\infty$ &  & 0.10 &  & $\infty$ &  & 0.27 \\
Robust yzflat & 0.11 & 0.09 & 0.11 & 0.40 &  &  &  &  &  &  \\
Robust line left & 0.16 & 0.02 & 0.16 & 0.97 &  &  &  &  &  &  \\
Robust line right & 0.15 & 0.02 & 0.17 & 1.14 &  &  &  &  &  &  \\
Robust fixed slope pairwise & 0.06 & 0.01 & 0.10 & 0.40 &  &  &  &  &  &  \\
Robust envelope & 0.04 & 0.01 & 0.06 & 0.38 &  &  &  &  &  &  \\
\midrule
Coefficient-wise & 0.18 & 0.19 & 0.15 & $\infty$ & 0.24 & 0.13 & 0.24 & $\infty$ & 0.12 & 0.28 \\
\midrule
Lagrangian bisegment & 0.16 & $\infty$ & $\infty$ & $\infty$ & 1.06 &  & 5.88 &  & 0.25 &  \\
Lagrangian bisegment coef & 0.04 & 0.01 & 0.01 & $\infty$ & 0.35 &  & 0.16 &  & 0.12 &  \\
\bottomrule
\end{tabular}
    }
    \caption{Median normalized RMSE of the bounds, with $N=10$ subintervals ($\infty$: the bound is not available at any point for at least half of the instances). }
    \label{tab:error}
\end{table}
\begin{table}[ht]
    \centering
    \resizebox{\textwidth}{!}{%
\begin{tabular}{lrrrrrrrrrr}
\toprule
 & \multicolumn{4}{c}{Netlib (LP)} & \multicolumn{4}{c}{MIPLIB (MILP)} & \multicolumn{2}{c}{Applications} \\
 & \multicolumn{2}{c}{ineq (53)} & \multicolumn{2}{c}{eq (68)} & \multicolumn{2}{c}{ineq (21)} & \multicolumn{2}{c}{eq (5)} & \multicolumn{2}{c}{(MILP, 126)} \\
\cmidrule(lr){2-3}\cmidrule(lr){4-5}\cmidrule(lr){6-7}\cmidrule(lr){8-9}\cmidrule(lr){10-11}
 & lb & ub & lb & ub & lb & ub & lb & ub & lb & ub \\
\midrule
Robust flat & 11.2 & 13.5 & 11.5 & 5.7 &  & 10.0 &  & 0.5 &  & 10.3 \\
Robust yzflat & 53.9 & 72.7 & 58.4 & 67.6 &  &  &  &  &  &  \\
Robust line left & 117.6 & 133.2 & 121.8 & 116.7 &  &  &  &  &  &  \\
Robust line right & 113.5 & 135.0 & 116.8 & 117.0 &  &  &  &  &  &  \\
Robust fixed slope pairwise & 69.2 & 82.7 & 75.2 & 79.5 &  &  &  &  &  &  \\
Robust envelope & 1596.2 & 1096.9 & 1859.3 & 512.2 &  &  &  &  &  &  \\
\midrule
Coefficient-wise & 12.3 & 12.0 & 15.0 & 6.2 & 10.6 & 9.9 & 11.5 & 0.5 & 10.1 & 10.2 \\
\midrule
Lagrangian bisegment & 21.0 & 21.8 & 17.2 & 19.2 & 33.4 &  & 1.0 &  & 8.7 &  \\
Lagrangian bisegment coef & 44.2 & 83.0 & 56.2 & 97.3 & 52.4 &  & 52.6 &  & 16.8 &  \\
\bottomrule
\end{tabular}
    }
    \caption{Median relative time to compute the bounds on the $N=10$ subintervals ($1$ = time to compute $f$ at one point).}
    \label{tab:time}
\end{table}

The bounding methods are compared in terms of availability, error and timing in Tables~\ref{tab:avail}, \ref{tab:error} and~\ref{tab:time} respectively, with $N=10$ subintervals. Let us recall that the coefficient-wise methods and Lagrangian bisegment coef need all the variables of the problem to be non-negative. Out of the $2\,939$ bound computations, two failed because of numerical issues of the solver (singular basis), one returned a lower bound slightly above $f$ at one point because of numerical tolerances, and seven reached the time limit; they are all counted as not available.

First, let us compare the methods in terms of \textbf{availability}. Lower bounds are available almost everywhere: the coefficient-wise relaxation is available on more than $96\%$ of the points on every dataset. Upper bounds are much less available on the instances with perturbed equality constraints (at most $64\%$ on Netlib eq and $40\%$ on MIPLIB eq). An upper bound on a subinterval requires a solution that is feasible for all $\lambda$ of the subinterval. For a perturbed equality row $(\mathbf{a}_i + \lambda \mathbf{d}_i)^t \mathbf{x} = b_i$, a constant solution must satisfy $\mathbf{d}_i^t \mathbf{x} = 0$, whatever the length of the subinterval: splitting the interval does not help, and the availability of the constant upper bounds (robust flat and coefficient-wise) is the same for $N = 1$, $5$ and $10$. In contrast, the variable robust solutions $\mathbf{y} + \lambda \mathbf{z}$ have more freedom, which explains their higher availability ($64\%$ instead of $25\%$ on Netlib eq, and $41\%$ with $N=1$). On the instances with perturbed inequality constraints, the robust and coefficient-wise upper bounds are available on more than $93\%$ of the points, except on the application instances ($87\%$ to $89\%$). Without the coefficient-wise constraints, the Lagrangian bisegment bound is often unavailable on LPs ($8\%$ to $63\%$), because its relaxation is unbounded for the chosen multipliers; adding these constraints makes it available on $88\%$ to $97\%$ of the points, except for the upper bounds on Netlib eq. On the subset of instances whose problem is infeasible for some values of $\lambda$ (Section~\ref{sec:dataset}), the lower bounds remain available on most feasible points (coefficient-wise: $95\%$ on Netlib and $100\%$ on MIPLIB), whereas the upper bounds are available on at most $70\%$ of them, since an upper bound is necessarily infinite on a subinterval that contains an infeasible value of $\lambda$.

Second, we compare the methods in terms of \textbf{error}. On LPs, the most precise bounds are the robust envelope and the Lagrangian bisegment coef bound, with a median error of at most $0.06$, except for the upper bounds on Netlib eq. The coefficient-wise bounds are less precise ($0.15$ to $0.19$, except for the upper bounds on Netlib eq), since they ignore that all the coefficients depend on the same $\lambda$. The robust flat lower bound, obtained from the dual problem, is loose (median error above $5$). On MILPs, the coefficient-wise bounds have a median error between $0.12$ and $0.28$ (except for the upper bounds on MIPLIB eq), close to that of the robust flat upper bound, and the Lagrangian bisegment coef lower bound between $0.12$ and $0.35$; without the coefficient-wise constraints, the Lagrangian lower bound is loose on the MIPLIB instances. Given the variation of $f$ (Tables~\ref{tab:dataset} and~\ref{tab:dataset_app}), a normalized error of $0.1$ to $0.3$ corresponds, for the coefficient-wise bounds, to a median error of about $1\%$ to $4\%$ of $|f|$.

Third, we compare the methods in terms of \textbf{timing}. The robust flat and coefficient-wise bounds have a median relative time between $10$ and $15$, i.e. about one solve of a problem of the size of the nominal one per subinterval. The lower values for the upper bounds on the equality instances come from bounding problems that are quickly found infeasible. The variable robust bounds, whose bounding problems have twice as many variables, are $5$ to $12$ times slower, and the robust envelope, which enumerates bases with the parametric simplex method, is about $100$ times slower. The Lagrangian bounds have a relative time between $1$ and $97$.

\paragraph*{Influence of the generation parameters.} To check that these conclusions do not depend on the choice of $R = 100$ perturbed rows and $3$ coefficients per row, we generated, from 20 of the problems (15 from Netlib and 5 from MIPLIB), instances in which one of these two numbers varies while the other keeps its reference value: $R \in \{10, 30, 100, 300\}$ (capped as in Section~\ref{sec:dataset}), and $c \in \{1, 2, 3, 5\}$ or all the nonzero coefficients of each perturbed row. The other generation choices and the filters are those of Section~\ref{sec:dataset}, which gives 287 instances; the results are given in Table~\ref{tab:robustness} in the appendix. The ranking of the methods is the same for all settings: the robust envelope and the Lagrangian bisegment coef bound remain more precise than the coefficient-wise bounds. The lower bounds, and the upper bounds on the instances with perturbed inequality constraints, remain available on at least $84\%$ of the points. As expected, the bounds tend to become looser when more coefficients are perturbed (e.g. from $0.09$ to $0.28$ for the coefficient-wise lower bound on the inequality instances, from one to all the coefficients of each row), and the upper bounds on the instances with perturbed equality constraints become less available (from $92\%$ to $25\%$ for the robust envelope, from $R=10$ to $R=300$), since each perturbed equality row further restricts the solutions that are feasible on a whole subinterval.

\paragraph*{MILP instances.} On MILPs, five of the methods apply: the coefficient-wise relaxation and restriction, the robust flat upper bound and the Lagrangian bisegment lower bounds. On the $126$ application instances, the coefficient-wise bounds are available on $99.8\%$ (lower bound) and $88.9\%$ (upper bound) of the points, with a median error of $0.12$ and $0.28$, for a relative time of about $10$, i.e. the cost of solving one MILP per subinterval. The results on the $26$ MIPLIB instances are similar, except for the upper bounds on the instances with perturbed equality constraints.

\subsection{Rank of the perturbation}\label{sec:rank}
The Sherman--Morrison and Woodbury formulas are efficient when the perturbation has a low rank (Section~\ref{sec:related}). We study whether the bounds depend on this rank, first on a dedicated synthetic dataset, then on the application dataset.

In the synthetic dataset, only $k$ columns of the constraint matrix are perturbed, with one common factor per column: each perturbed coefficient $a_{ij}$ of column $j$ becomes $a_{ij}(1 + \lambda r_j)$, with $|r_j|$ drawn uniformly in $[0.1, 0.5]$ and a random sign, so that $\text{rank}(D) \leq k$, for $k \in \{1, 2, 3, 5, 10, 20\}$. Either a fixed number of coefficients (between 34 and 50, depending on the problem) is spread over the $k$ columns, or a fixed fraction ($50\%$) of the nonzero coefficients of each column is perturbed, so that their number grows with $k$. The instances are derived from three Netlib problems and one MIPLIB problem, with three instances per problem, setting and value of $k$, and the filters of Section~\ref{sec:dataset} (133 instances). With $k = 1$ and a fixed number of coefficients, all the LP candidates have a flat $f$, so this case is left out.

\begin{table}[ht]
    \centering
    \resizebox{\textwidth}{!}{%
\begin{tabular}{lrrrrrrrrrrr}
\toprule
 & \multicolumn{5}{c}{Fixed number of coefficients} & \multicolumn{6}{c}{Fixed fraction of each column} \\
\cmidrule(lr){2-6}\cmidrule(lr){7-12}
$k$ & 2 & 3 & 5 & 10 & 20 & 1 & 2 & 3 & 5 & 10 & 20 \\
\midrule
\multicolumn{12}{l}{\textit{Lower bounds}} \\
Robust envelope & 0.01 & 0.02 & 0.02 & 0.01 & 0.02 & 0.01 & 0.02 & 0.01 & 0.01 & 0.01 & 0.01 \\
Coefficient-wise & 0.24 & 0.14 & 0.11 & 0.11 & 0.12 & 0.10 & 0.11 & 0.10 & 0.10 & 0.12 & 0.12 \\
Lagrangian bisegment coef & 0.01 & 0.01 & 0.01 & 0.01 & 0.01 & 0.01 & 0.01 & 0.01 & 0.01 & 0.01 & 0.01 \\
\midrule
\multicolumn{12}{l}{\textit{Upper bounds}} \\
Robust envelope & 0.00 & 0.01 & 0.01 & 0.00 & 0.00 & 0.00 & 0.00 & 0.01 & 0.01 & 0.00 & 0.00 \\
Coefficient-wise & 0.42 & 0.40 & 0.43 & 0.41 & 0.42 & 0.36 & 0.38 & 0.39 & 0.45 & 0.51 & 1.14 \\
Lagrangian bisegment coef & 0.01 & 0.00 & 0.01 & 0.01 & 0.01 & 0.01 & 0.01 & 0.01 & 0.02 & 0.01 & 0.02 \\
\bottomrule
\end{tabular}
    }
    \caption{Median normalized RMSE of three of the bounds as a function of the number $k$ of perturbed columns ($\text{rank}(D) \leq k$), with $N=10$ subintervals. Each column contains 9 LP instances and 3 MILP instances (the robust bounds and the Lagrangian upper bounds are computed on the LPs only), except for $k=2$ with a fixed number of coefficients (7 LP instances).}
    \label{tab:rank}
\end{table}

The error of each bound is about the same for $k = 1$ or $2$ and for $k = 20$ (Table~\ref{tab:rank}), and so is its availability (not shown); the other bounds of Table~\ref{tab:all_bounds} show no trend with $k$ either. The only exception is the coefficient-wise upper bound (which coincides with the robust flat one on these instances): its availability decreases with $k$ (from $100\%$ to $75\%$ and from $92\%$ to $67\%$), and its error grows in the second setting (from $0.36$ to $1.14$), where the number of perturbed coefficients grows with $k$.

On the application dataset, the rank of $D$ is never one: it ranges from 3 to 1430, with a median of 25. The error of the coefficient-wise bounds shows no trend with this rank (Spearman correlation of $-0.10$ for the lower bound and $0.14$ for the upper bound, not significant); the error varies mostly between application domains. Unlike the Sherman--Morrison and Woodbury formulas, the bounds therefore do not exploit a low rank of the perturbation, but they do not suffer from a high rank either.

\subsection{Adaptive refinement algorithm}\label{sec:opt_combination}
For the adaptive refinement algorithm, we consider the coefficient-wise bounding methods for both upper and lower bounds, since they apply to LPs and MILPs and are among the cheapest bounds (Table~\ref{tab:time}). The algorithm is run on the Netlib and MIPLIB instances with perturbed inequality constraints and on the application instances, with a time limit equivalent to 250 evaluations of $f$ (and at most one hour). We do not consider the instances with perturbed equality constraints, because of the limitation explained in Section~\ref{sec:bound_assess}: on these instances, the coefficient-wise upper bound does not improve when the interval is split. The minimal interval length below which the algorithm computes $f$ at the middle of the interval instead of refining it further is set to $\epsilon_\lambda = (\lambda_2-\lambda_1)/100$. 

\begin{figure}[h!]
    \centering
    \includegraphics[width=0.65\textwidth]{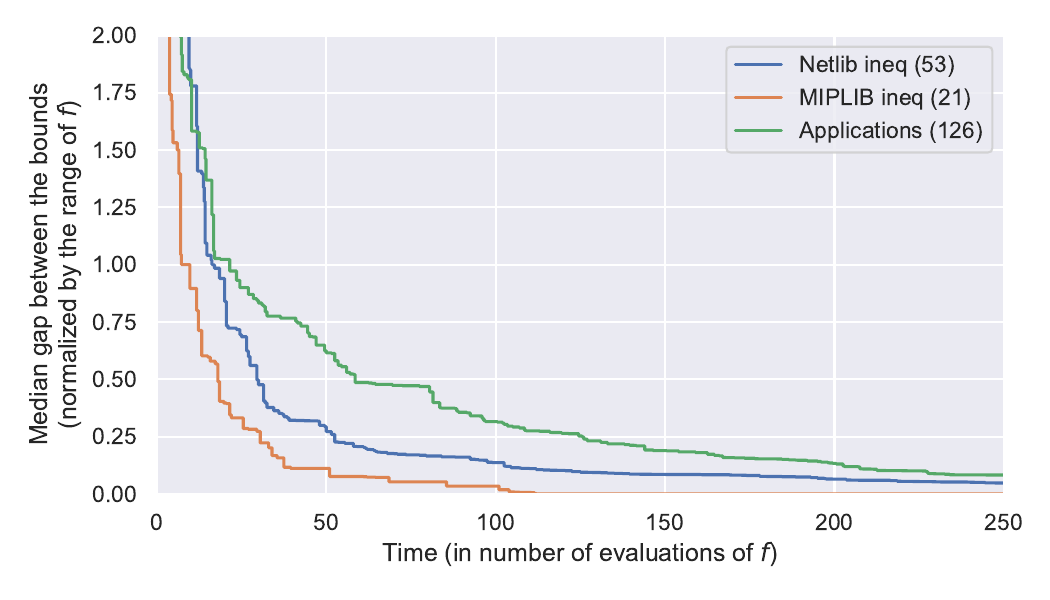}
    \caption{Median gap between the best lower and upper bounds (normalized by the range of $f$) as a function of the time (in number of evaluations of $f$) for the adaptive refinement algorithm with the coefficient-wise relaxation and restriction. The number of instances is given in parentheses.}
    \label{fig:it_alg_results}
\end{figure}

We show how the gap between the bounds evolves in terms of  equivalent time to compute one point of $f$. At a given time, the best lower (resp. upper) bound at $\lambda$ is the largest (resp. smallest) of the bounds computed so far on the intervals that contain $\lambda$. For each problem instance, we compute the  median, over $1\,000$ equally spaced values of $\lambda$, of the difference between the best upper and lower bounds, normalized by the range of $f$ as in Section~\ref{sec:dataset}. The overall reported gap corresponds to the median of these per-problem gaps. To compute an instance-independent time, we  report the timing $t$ as the current timing $t_c$ divided by $t_m$, the mean time to compute $f$ at one point (Section~\ref{sec:bound_assess}). The results are shown in Figure \ref{fig:it_alg_results}. Unlike the errors of Section~\ref{sec:bound_assess}, this gap is known without computing $f$, and $f$ is guaranteed to lie between the bounds. 

 On the instances with perturbed inequality constraints and on the application instances, the median gap decreases steadily: after a computation time equivalent to $100$ evaluations of $f$, it is $14\%$ of the range of $f$ for Netlib ineq, $3\%$ for MIPLIB ineq and $32\%$ for the application dataset; after $250$ evaluations, it is $5\%$, $0\%$ and $8\%$.  

%% file: TexFiles/conclusion.tex
\section{Conclusion}\label{sec:conclusion}

In this paper, we propose a novel approach for assessing the behaviour of the optimal objective function of a problem whose constraint coefficients can linearly vary within a given interval $[\lambda_1, \lambda_2]$. This approach consists of building bounds around the optimal objective function. Building bounds provide strong guarantees on the objective function and helps identify problematic behaviours in the objective that can otherwise be missed.

We present bounding methods based on robust optimization, coefficient-wise reformulation and Lagrangian relaxations. The robust methods are derived in three groups of methods, one that computes constant bounds, the other variable bounds depending on $\lambda$ and finally, envelope bounds. We benchmark all methods and highlight the efficiency of the bounding approach in terms of precision, availability and timing. For most linear and MILP problems, the bounding methods provide a high precision for little computational time. In particular, the coefficient-wise bounding method often offers an interesting tradeoff between precision and timing. However, we also showed that when the parameter $\lambda$ impacts equality constraints, upper bounds are difficult to obtain, for LPs and MILPs alike, leaving room for improvement and further research for this class of problems. A summary of when to use each bounding method is provided in Figure \ref{fig:decision_tree}. 

\begin{figure}[H]
    \centering
    \input{Images/decision_tree}
    \caption{Overview of the decision tree outlining when to use each method}
    \label{fig:decision_tree}
\end{figure}

We propose an adaptive refinement algorithm. It uses a lower and an upper bounding method and iteratively refines the gap between the two by dividing the interval and reapplying the methods. We benchmark using the coefficient-wise method to compute upper and lower bounds and showed that, on the problems with perturbed inequality constraints and on the application dataset, the median gap between the bounds is between $3\%$ and $32\%$ of the range of $f$ after a computation time equivalent to 100 evaluations of $f$, and at most $8\%$ after 250 evaluations.

%% file: Images/decision_tree.tex
\begin{tikzpicture}[
    >={Stealth[length=2mm]},
    every node/.style={font=\small},
    question/.style={draw=black!55, fill=black!5, rounded corners=3pt, align=center, text width=2.5cm,
                     execute at begin node={\hyphenpenalty=10000\exhyphenpenalty=10000},
                     minimum height=9mm, inner sep=3pt},
    method/.style={draw=blue!60!black, fill=blue!8, text=blue!40!black, rounded corners=3pt, align=center, text width=2.6cm,
                   execute at begin node={\hyphenpenalty=10000\exhyphenpenalty=10000},
                   minimum height=9mm, inner sep=3pt},
    yes/.style={->, thick, draw=green!45!black},
    no/.style={->, thick, draw=orange!85!black},
    lbl/.style={font=\scriptsize, fill=white, inner sep=1.5pt},
]
    \node[question, text width=3.6cm] (eq) at (0, 0) {Perturbed equality constraints?};
    \node[question] (lp)    at (-4.2, -1.8) {LP?};
    \node[question] (nnL)   at (-2.4, -3.7) {Non-negative variables?};
    \node[question] (nnR)   at (3.9, -1.8)  {Non-negative variables?};
    \node[question] (milp)  at (5.6, -3.7)  {MILP?};
    \node[question] (prec)  at (3.1, -5.6)  {High precision needed?};

    \node[method] (slope)   at (-6.0, -3.7) {Robust fixed slope pairwise};
    \node[method] (coef)    at (-4.0, -5.6) {Coefficient-wise};
    \node[method] (flat)    at (-0.8, -5.6) {Robust flat};
    \node[method] (coeflag) at (1.9, -3.7)  {Coefficient-wise or Lagrangian bisegment coef};
    \node[method] (flatlag) at (6.7, -5.6) {Robust flat or Lagrangian bisegment};
    \node[method] (var)     at (3.1, -7.5)  {Robust variable methods or robust envelope};

    \draw[yes] (eq)   -- node[lbl] {yes} (lp);
    \draw[no]  (eq)   -- node[lbl] {no}  (nnR);
    \draw[yes] (lp)   -- node[lbl] {yes} (slope);
    \draw[no]  (lp)   -- node[lbl] {no}  (nnL);
    \draw[yes] (nnL)  -- node[lbl] {yes} (coef);
    \draw[no]  (nnL)  -- node[lbl] {no}  (flat);
    \draw[yes] (nnR)  -- node[lbl] {yes} (coeflag);
    \draw[no]  (nnR)  -- node[lbl] {no}  (milp);
    \draw[yes] (milp) -- node[lbl] {yes} (flatlag);
    \draw[no]  (milp) -- node[lbl] {no}  (prec);
    \draw[no]  (prec) -- node[lbl] {no}  (flatlag);
    \draw[yes] (prec) -- node[lbl] {yes} (var);
\end{tikzpicture}

%% file: TexFiles/applications.tex
\subsection{Lemma quadratic function}\label{annex:lemma}
This Lemma is used in Section \ref{sec:variable}.

\begin{lemma}
Given a quadratic function $q(x)=ax^2+bx+c$. The maximum of $q(x)$ over $x_1 \leq x \leq x_2$ is upper bounded by 
\begin{align}
    \max \begin{cases}
        ax_1^2+bx_1+c\\
        ax_2^2+bx_2+c\\
        ax_1x_2 + b\frac{x_1+x_2}{2} + c
    \end{cases}
\end{align}
\label{lemma:quadratic-tangents}
\end{lemma}
\begin{proof}
    We analyse separately the concave and the convex case:
    \begin{itemize}
        \item If the function $q(x)$ is convex, that is if $a \geq 0$, then the maximum of the function in the polyhedron $[x_1,x_2]$. The equation \eqref{lemma:quadratic-tangents} is known to be attained at one of its vertices, either $x_1$ or $x_2$.
        \item For the concave case, let us consider the two tangents at $x_1$ and $x_2$:
    \begin{align}
        t_{x_1}(x) = ax_1^2+bx_1+c + (x-x_1) (2ax_1+b) = c + 2ax_1x - ax_1^2 + bx\\
        t_{x_2}(x) = ax_2^2+bx_2+c + (x-x_2) (2ax_2+b) = c + 2ax_2x - ax_2^2 + bx.
    \end{align}
    If $q(x)$ is concave ($a \leq 0$) then any tangent is an upper bound for the function for any $x$. Taking two tangent at different points and taking the minimum of the two functions is also an upper bound. That is, for any $x_1 < x_2$:
            \begin{align}
                \min(t_{x_1}(x), t_{x_2}(x)) \geq q(x)\quad \forall x.
            \end{align}
            $\min(t_{x_1}(x), t_{x_2}(x))$ forms a piecewise linear function. The tangents $t_{x_1}$ and $t_{x_2}$ intersect at the middle point of $[x_1, x_2]$:
            \begin{align}
                &\quad t_{x_1}(x) = t_{x_2}(x)\\
                & \quad 2ax_1x - ax_1^2 = 2ax_2x - ax_2^2\\
                & \quad  2x_1x - x_1^2 = 2x_2x - x_2^2\\
                & \quad x = \frac{x_2^2 - x_1^2}{2(x_2-x_1)} = \frac{x_1+x_2}{2}
            \end{align}
            with the value $t_1(\frac{x_1+x_2}{2}) = ax_1x_2 + b\frac{x_1+x_2}{2} + c$. 
            
            Now, the maximum of the function $\min(t_{x_1}(x), t_{x_2}(x))$ over $[x_1, x_2]$ is achieved either on $t_{x_1}(x_1)$ or $t_{x_2}(x_2)$ or in the middle of the range, at $t_1(\frac{x_1+x_2}{2})$:
            \begin{align}
                \max_{x\in[x_1,x_2]} \min(t_{x_1}(x), t_{x_2}(x)) = \max \begin{cases}
                    ax_1^2+bx_1+c\\
                    ax_2^2+bx_2+c\\
                    ax_1x_2 + b\frac{x_1+x_2}{2} + c
                \end{cases}\label{max_q}
            \end{align}
    \end{itemize}
\end{proof}

\subsection{Dataset}\label{annex:examples}
In this Appendix, we provide a full overview the illustrative problems. 

\subsubsection{Illustrative problems}
This category contains three small two-variable problems that are used for the sake of illustration rather than being of particular interest. 

\begin{enumerate}
    \item \textbf{Toy 1:}
\begin{align}\label{ex:toy_1}
\begin{split}
    \mathcal{P}_{\text{toy 1}}(\lambda) \equiv \min\quad &\begin{pmatrix}-1  & 2\end{pmatrix} \begin{pmatrix}
        x \\ y
    \end{pmatrix}\\
    \text{ s.t } \quad& \begin{pmatrix}
        -2 & -1 \\
        1 & 2\\
        -1 & 1\\
        1 & -1 \\
        1 & -3\\
        2 & 0
    \end{pmatrix}
    \begin{pmatrix}
        x \\ y
    \end{pmatrix}
    \leq \begin{pmatrix}
        3\\ 0\\ 2\\ 1\\ 1 \\3
    \end{pmatrix}\\
    & \begin{pmatrix}
        3 & -1 \\ -3 & 1 \\ -2 & 1 \\ 2 & -1\\ -1 & -2 \\ -1 & -3
    \end{pmatrix}\begin{pmatrix}
        x \\ y
    \end{pmatrix}+ \lambda \begin{pmatrix}
        3 & 3 \\ 3 & -2 \\ 2 & 3 \\ 2 & -1 \\ -2 & 1 \\ 0 & -2
    \end{pmatrix}\begin{pmatrix}
        x \\ y
    \end{pmatrix}
    \leq 
    \begin{pmatrix}
        0\\ 0\\ 0\\ 3\\ 3\\ 1
    \end{pmatrix}\\
    & \forall \lambda \in [-4, 4]
\end{split}
\end{align}
    \item \textbf{Toy 2:}
    \begin{align}\label{ex:toy_2}
\begin{split}
    \mathcal{P}_{\text{toy 2}}(\lambda) \equiv \min\quad &\begin{pmatrix}0  & 1\end{pmatrix} \begin{pmatrix}
        x \\ y
    \end{pmatrix}\\
    \text{ s.t } \quad& \begin{pmatrix}
        -2 & 0 \\
        2 & 2\\
        -1 & 0\\
        -1 & -1 \\
        0 & 1\\
        -3 & -2
    \end{pmatrix}
    \begin{pmatrix}
        x \\ y
    \end{pmatrix}
    \leq \begin{pmatrix}
        1\\ 2\\ 2\\ 3\\ 3 \\1
    \end{pmatrix}\\
    & \begin{pmatrix}
        2 & 3 \\ 0 & 1 \\ 1 & -3 \\ 0 & -3\\ 3 & 0 \\ -2 & 3
    \end{pmatrix}\begin{pmatrix}
        x \\ y
    \end{pmatrix}+ \lambda \begin{pmatrix}
        3 & 3 \\ 0 & 2 \\ -2 & -1 \\ 2 & 1 \\ 2 & 0\\2 & 2
    \end{pmatrix}\begin{pmatrix}
        x \\ y
    \end{pmatrix}
    \leq 
    \begin{pmatrix}
        3\\ 1\\ 2\\ 3\\ 3\\ 1
    \end{pmatrix}\\
    & \forall \lambda \in [-4, 4]
\end{split}
\end{align}
    \item \textbf{Toy 3:}
    This problem is the one written in \eqref{example:toy}. We rewrite it here for the sake of completeness.
    \begin{align*}
\begin{split}
    \mathcal{P}_{\text{toy 3}}(\lambda) \equiv \min\quad &\begin{pmatrix}2  & -2\end{pmatrix} \begin{pmatrix}
        x \\ y
    \end{pmatrix}\\
    \text{ s.t } \quad& \begin{pmatrix}
        -2 & 2 \\
        -1 & 0
    \end{pmatrix}
    \begin{pmatrix}
        x \\ y
    \end{pmatrix}
    \leq \begin{pmatrix}
        4 \\ 1
    \end{pmatrix}\\
    & \begin{pmatrix}
        2 & 1 \\ -2 & -3 \\ 2 & 2 \\ -1 & -4
    \end{pmatrix}\begin{pmatrix}
        x \\ y
    \end{pmatrix} + \lambda \begin{pmatrix}
        -1 & -4 \\ 0 & 4 \\ -4 & -3 \\ 2 & 4
    \end{pmatrix}\begin{pmatrix}
        x \\ y
    \end{pmatrix}
    \leq 
    \begin{pmatrix}
        4 \\ 2 \\ 0 \\ 2
    \end{pmatrix}\\
    & \forall \lambda \in [-10, 9]
\end{split}
\end{align*}
\item \textbf{Toy 4: } This problem is the one written in \eqref{eq:example-robust-1}. We rewrite it for the sake of completeness.
\begin{align}
\begin{split}
    \mathcal{P}_{\text{toy 4}}(\lambda) \equiv \min_{x, y}\quad &\begin{pmatrix}
        -2 & -2
    \end{pmatrix}
    \begin{pmatrix}
        x \\ y
    \end{pmatrix}\\
    \text{ s.t } \quad &\begin{pmatrix}
        3 & 1
    \end{pmatrix}\begin{pmatrix}
        x \\ y
    \end{pmatrix} \leq \begin{pmatrix}
        3
    \end{pmatrix}\\
    &\begin{pmatrix}
        -5 & -2 \\
        1 & 4
    \end{pmatrix}
    \begin{pmatrix}
        x \\ y
    \end{pmatrix}+\lambda \begin{pmatrix}
        -3 & -2 \\
        -3 & 0
    \end{pmatrix} \begin{pmatrix}
        x \\ y
    \end{pmatrix} \leq
    \begin{pmatrix}
        0 \\ -3
    \end{pmatrix}\\
    & \forall \lambda \in [-2, 2]
\end{split}
\end{align}
\end{enumerate}

\subsection{Influence of the generation parameters}\label{annex:robustness}
\begin{table}[ht]
    \centering
    \resizebox{\textwidth}{!}{%
\begin{tabular}{lcccccccc}
\toprule
 & \multicolumn{4}{c}{ineq} & \multicolumn{4}{c}{eq} \\
\cmidrule(lr){2-5}\cmidrule(lr){6-9}
 & Coef.-wise lb & Coef.-wise ub & Envelope ub & Lag. bis. coef lb & Coef.-wise lb & Coef.-wise ub & Envelope ub & Lag. bis. coef lb \\
\midrule
$R=10$ & 100 / 0.22 & 98 / 0.13 & 100 / 0.01 & 100 / 0.02 & 93 / 0.11 & 27 / $\infty$ & 92 / 0.02 & 99 / 0.01 \\
$R=30$ & 100 / 0.18 & 98 / 0.16 & 100 / 0.01 & 100 / 0.02 & 93 / 0.15 & 20 / $\infty$ & 75 / 0.18 & 100 / 0.02 \\
$R=100$ & 100 / 0.18 & 92 / 0.13 & 100 / 0.02 & 100 / 0.04 & 100 / 0.22 & 7 / $\infty$ & 49 / $\infty$ & 93 / 0.02 \\
$R=300$ & 100 / 0.23 & 89 / 0.20 & 99 / 0.02 & 95 / 0.06 & 100 / 0.27 & 7 / $\infty$ & 25 / $\infty$ & 93 / 0.02 \\
\midrule
$c=1$ & 93 / 0.09 & 100 / 0.09 & 100 / 0.01 & 100 / 0.02 & 100 / 0.13 & 24 / $\infty$ & 72 / 0.27 & 100 / 0.02 \\
$c=2$ & 93 / 0.12 & 99 / 0.11 & 100 / 0.01 & 100 / 0.03 & 100 / 0.25 & 18 / $\infty$ & 54 / 2.84 & 100 / 0.01 \\
$c=3$ & 93 / 0.18 & 98 / 0.16 & 100 / 0.01 & 93 / 0.06 & 100 / 0.22 & 12 / $\infty$ & 61 / 1.50 & 94 / 0.02 \\
$c=5$ & 93 / 0.17 & 99 / 0.16 & 100 / 0.01 & 89 / 0.04 & 100 / 0.28 & 12 / $\infty$ & 46 / $\infty$ & 94 / 0.11 \\
all & 93 / 0.28 & 96 / 0.34 & 100 / 0.01 & 84 / 0.19 & 94 / 0.27 & 6 / $\infty$ & 58 / 0.26 & 93 / 0.05 \\
\bottomrule
\end{tabular}
    }
    \caption{Influence of the number $R$ of perturbed rows (top, with $c = 3$ coefficients per row) and of the number $c$ of perturbed coefficients per row (bottom, with $R = 100$) on the bounds, with $N=10$ subintervals: mean availability (in \%) / median normalized RMSE. For each of the two factors, only the problems and constructions for which all the settings were accepted are kept (top: 11 ineq and 15 eq; bottom: 15 ineq and 17 eq). The robust envelope is computed on the LPs only.}
    \label{tab:robustness}
\end{table}

%% file: main.bbl
\begin{thebibliography}{10}

\bibitem{Adler}
Ilan Adler and Renato D.~C. Monteiro.
\newblock A geometric view of parametric linear programming.
\newblock {\em Algorithmica}, 8:161–176, 1992.

\bibitem{Anderson2023}
{Kim Allan} Andersen, {Trine Krogh} Boomsma, and {Lars Relund} Nielsen.
\newblock Milp sensitivity analysis for the objective function coefficients.
\newblock {\em INFORMS Journal on Optimization}, 5(1):92--109, 2023.

\bibitem{beasley2003indtrack}
John~E. Beasley, Nigel Meade, and T.-J. Chang.
\newblock An evolutionary heuristic for the index tracking problem.
\newblock {\em European Journal of Operational Research}, 148(3):621--643,
  2003.

\bibitem{Berkelaar1997}
Arjan~B. Berkelaar, Kees Roos, and Tam{\'a}s Terlaky.
\newblock {\em The Optimal Set and Optimal Partition Approach to Linear and
  Quadratic Programming}, pages 159--202.
\newblock Springer US, Boston, MA, 1997.

\bibitem{bertsimas}
Dimitris Bertsimas and John~N. Tsitsiklis.
\newblock {\em Introduction to linear optimization}.
\newblock Athena Scientific, 1997.

\bibitem{brown2018pypsa}
Thomas Brown, Jonas H\"orsch, and David Schlachtberger.
\newblock {PyPSA}: Python for power system analysis.
\newblock {\em Journal of Open Research Software}, 6(1):4, 2018.

\bibitem{charitopoulos2018mpmilp}
Vassilis~M. Charitopoulos, Lazaros~G. Papageorgiou, and Vivek Dua.
\newblock Multi-parametric mixed integer linear programming under global
  uncertainty.
\newblock {\em Computers \& Chemical Engineering}, 116:279--295, 2018.

\bibitem{chinneck2000interval}
John~W. Chinneck and Kamal Ramadan.
\newblock Linear programming with interval coefficients.
\newblock {\em Journal of the Operational Research Society}, 51(2):209--220,
  2000.

\bibitem{crema2002parametric}
Alejandro Crema.
\newblock An algorithm to perform a complete parametric analysis relative to
  the constraint matrix for a 0--1-integer linear program.
\newblock {\em European Journal of Operational Research}, 138(3):484--494,
  2002.

\bibitem{Dawande2000}
M.~W. Dawande and John~N. Hooker.
\newblock {Inference-Based Sensitivity Analysis for Mixed Integer/Linear
  Programming}.
\newblock 8 2000.

\bibitem{derval2025warmstart}
Guillaume Derval, Bardhyl Miftari, Damien Ernst, and Quentin Louveaux.
\newblock Efficient {LP} warmstarting for linear modifications of the
  constraint matrix.
\newblock arXiv preprint arXiv:2501.04151, 2025.
\newblock \url{https://arxiv.org/abs/2501.04151}.

\bibitem{ernst1996hub}
Andreas~T. Ernst and Mohan Krishnamoorthy.
\newblock Efficient algorithms for the uncapacitated single allocation p-hub
  median problem.
\newblock {\em Location Science}, 4(3):139--154, 1996.

\bibitem{ernst1999hub}
Andreas~T. Ernst and Mohan Krishnamoorthy.
\newblock Solution algorithms for the capacitated single allocation hub
  location problem.
\newblock {\em Annals of Operations Research}, 86:141--159, 1999.

\bibitem{feng2022spopt}
Xin Feng, Germano Barcelos, James~D. Gaboardi, Elijah Knaap, Ran Wei, Levi~J.
  Wolf, Qunshan Zhao, and Sergio~J. Rey.
\newblock spopt: a python package for solving spatial optimization problems in
  {PySAL}.
\newblock {\em Journal of Open Source Software}, 7(74):3330, 2022.

\bibitem{flavell1975approach}
Richard Flavell and Gerald~R. Salkin.
\newblock An approach to sensitivity analysis.
\newblock {\em Journal of the Operational Research Society}, 26:857--866, 1975.

\bibitem{gal_multi}
Thomas Gal and Josef Nedoma.
\newblock Multiparametric linear programming.
\newblock {\em Management Science}, 7(18):406--422, 1972.

\bibitem{Gal-1994}
Tomas Gal.
\newblock {\em Postoptimal Analyses, Parametric Programming, and Related
  Topics, Degeneracy, Multicriteria Decision Making, Redundancy}.
\newblock De Gruyter, New York, 1994.

\bibitem{gass1955computational}
Saul Gass and Thomas Saaty.
\newblock The computational algorithm for the parametric objective function.
\newblock {\em Naval Research Logistics Quarterly}, 2(1-2):39--45, 1955.

\bibitem{obj_param2}
Saul~I. Gass and Thomas~L. Saaty.
\newblock Parametric objective function (part 2)- generalization.
\newblock {\em Journal of the Operations Research Society of America},
  3(4):395--401, 1955.

\bibitem{netlib}
David~M. Gay.
\newblock Electronic mail distribution of linear programming test problems.
\newblock {\em Mathematical Programming Society COAL Newsletter}, 13:10--12,
  1985.

\bibitem{Geoffrion}
Arthur~M. Geoffrion and Robert~M. Nauss.
\newblock Parametric and postoptimality analysis in integer linear programming.
\newblock {\em Management Science}, 23(5):453--466, 1977.

\bibitem{miplib}
Ambros Gleixner, Gregor Hendel, Gerald Gamrath, Tobias Achterberg, Michael
  Bastubbe, Timo Berthold, Philipp~M. Christophel, Kati Jarck, Thorsten Koch,
  Jeff Linderoth, Marco L\"ubbecke, Hans~D. Mittelmann, Derya Ozyurt, Ted~K.
  Ralphs, Domenico Salvagnin, and Yuji Shinano.
\newblock {MIPLIB 2017: Data-Driven Compilation of the 6th Mixed-Integer
  Programming Library}.
\newblock {\em Mathematical Programming Computation}, 2021.

\bibitem{jansen}
Benjamin Jansen, Johan~J. {de Jong}, Cees Roos, and Tamas Terlaky.
\newblock Sensitivity analysis in linear programming: just be careful!
\newblock {\em European Journal of Operational Research}, 101(1):15--28, 1997.

\bibitem{Bowman01121972}
V.~Joseph~Bowman Jr.
\newblock Sensitivity analysis in linear integer programming.
\newblock {\em AIIE Transactions}, 4(4):284--289, 1972.

\bibitem{lhs_param_algo}
Rajab Khalilpour and Iftekar~A. Karimin.
\newblock Parametric optimization with uncertainty on the left hand side of
  linear programs.
\newblock {\em Computers \& Chemical Engineering}, 60:31--40, 2014.

\bibitem{knueven2020uc}
Bernard Knueven, James Ostrowski, and Jean-Paul Watson.
\newblock On mixed-integer programming formulations for the unit commitment
  problem.
\newblock {\em INFORMS Journal on Computing}, 32(4):857--876, 2020.

\bibitem{kondili1993stn}
E.~Kondili, C.~C. Pantelides, and R.~W.~H. Sargent.
\newblock A general algorithm for short-term scheduling of batch operations --
  {I}. {MILP} formulation.
\newblock {\em Computers \& Chemical Engineering}, 17(2):211--227, 1993.

\bibitem{li2007mpmilp}
Zukui Li and Marianthi~G. Ierapetritou.
\newblock A new methodology for the general multiparametric mixed-integer
  linear programming ({MILP}) problems.
\newblock {\em Industrial \& Engineering Chemistry Research},
  46(15):5141--5151, 2007.

\bibitem{medjroubi2017scigrid}
Wided Medjroubi, Ulf~Philipp M\"uller, Malte Scharf, Carsten Matke, and David
  Kleinhans.
\newblock Open data in power grid modelling: New approaches towards transparent
  grid models.
\newblock {\em Energy Reports}, 3:14--21, 2017.

\bibitem{zenodo_code}
Bardhyl Miftari and Guillaume Derval.
\newblock {Sensitivity analysis for linear changes of the constraint matrix of
  a (mixed-integer) linear program: code and datasets (revised version)},
  November 2026.

\bibitem{zenodo_output}
Bardhyl Miftari and Guillaume Derval.
\newblock {Sensitivity analysis for linear changes of the constraint matrix of
  a (mixed-integer) linear program: raw outputs (revised version)}, November
  2026.

\bibitem{10.12688/f1000research.29032.1}
Félix Möder, Kim~Philipp Jablonski, Brice Letcher, Michael~B. Hall,
  Christopher~H. Tomkins-Tinch, Vanessa Sochat, Jan Forster, Soohyun Lee,
  Sven~O. Twardziok, Alexander Kanitz, Andreas Wilm, Manuel Holtgrewe, Sven
  Rahmann, Sven Nahnsen, and Johannes Köter.
\newblock Sustainable data analysis with snakemake.
\newblock {\em F1000Research}, 10(33), 2021.

\bibitem{orlowski2010sndlib}
Sebastian Orlowski, Roland Wess\"aly, Micha{\l} Pi\'oro, and Artur Tomaszewski.
\newblock {SNDlib} 1.0 -- survivable network design library.
\newblock {\em Networks}, 55(3):276--286, 2010.

\bibitem{obj_param1}
Thomas Saaty and Saul Gass.
\newblock Parametric objective function (part 1).
\newblock {\em Journal of the Operations Research Society of America},
  2(3):316--319, 1954.

\bibitem{Sherman2}
Jack Sherman and Winifred~J. Morrison.
\newblock Adjustment of an inverse matrix corresponding to changes in the
  elements of a given column or a given row of the original matrix.
\newblock {\em The Annals of Mathematical Statistics}, 20:620--624, 1949.

\bibitem{Sherman}
Jack Sherman and Winifred~J. Morrison.
\newblock Adjustment of an inverse matrix corresponding to a change in one
  element of a given matrix.
\newblock {\em The Annals of Mathematical Statistics}, 21(1):124--127, 1950.

\bibitem{woodbury}
Max~A. Woodbury.
\newblock {\em Inverting Modified Matrices}.
\newblock Memorandum Report / Statistical Research Group, Princeton. Department
  of Statistics, Princeton University, 1950.

\bibitem{Zuidwijk}
Rob~A. Zuidwijk.
\newblock Linear parametric sensitivity analysis of the constraint coefficient
  matrix in linear programs.
\newblock {\em ERIM report series research in management}, ERS-2005-055-LIS,
  2005.

\end{thebibliography}
